\documentclass{article}
\usepackage[margin=1.2in]{geometry}
\usepackage[T1]{fontenc}
\usepackage[utf8]{inputenc}
\usepackage{lmodern}
\usepackage{microtype}
\usepackage{amsmath,amssymb,amsthm,amsfonts,mathtools,mathabx,bm}
\usepackage{tikz}
\usetikzlibrary{cd}
\usetikzlibrary{shapes}
\usetikzlibrary{shapes.multipart}
\usepackage{multicol}
\usepackage{array}
\usepackage[shortlabels]{enumitem}

\usepackage[hyphens]{url}
\usepackage{nicefrac}
\usepackage{subfiles}

\newcommand\ip[2]{\langle #1, #2 \rangle}
\newcommand\seq[1]{\langle #1 \rangle}

\DeclareMathOperator{\dom}{dom}
\DeclareMathOperator{\range}{range}

\DeclareMathOperator{\uh}{\upharpoonright}
\DeclareMathOperator{\Cone}{Cone}
\DeclareMathOperator{\Spec}{Spec}

\def\id{id}
\def\I{\mathcal{I}}
\def\N{\mathbb{N}}

\def\R{\mathbb{R}}

\def\Cantor{2^\omega}
\def\Ord{\textbf{Ord}}
\def\ZFC{\mathsf{ZFC}}
\def\ZF{\mathsf{ZF}}
\def\AD{\mathsf{AD}}

\def\DC{\mathsf{DC}}

\def\CC{\mathsf{CC}}
\def\Unif{\mathsf{Uniformization}}

\def\Det{\mathsf{Det}}

\def\D{\mathcal{D}}
\def\hyperjump{\mathcal{O}}

\def\degreeof{\mathop{\text{deg}_T}}

\let\tilde\widetilde
\let\degree\bm
\let\term\textbf
\let\0\varnothing

\let\phi\varphi

\renewcommand{\epsilon}{\varepsilon}
\renewcommand{\setminus}{\smallsetminus}

\theoremstyle{plain}
\newtheorem{theorem}{Theorem}
\newtheorem*{notation}{Notation}

\newtheorem*{claim*}{Claim}

\newtheorem{proposition}[theorem]{Proposition}
\newtheorem*{proposition*}{Proposition}

\newtheorem*{fact*}{Fact}

\newtheorem*{conjecture*}{Conjecture}
\newtheorem{corollary}[theorem]{Corollary}
\newtheorem{lemma}[theorem]{Lemma}
\newtheorem*{lemma*}{Lemma}

\newtheorem{question}{Question}
\newtheorem*{question*}{Question}
\theoremstyle{definition}\newtheorem{remark}[theorem]{Remark}
\theoremstyle{definition}\newtheorem*{remark*}{Remark}
\theoremstyle{definition}\newtheorem{definition}[theorem]{Definition}
\theoremstyle{definition}
\theoremstyle{definition}
\theoremstyle{definition}\newtheorem{example}[theorem]{Example}
\theoremstyle{definition}\newtheorem*{example*}{Example}

\numberwithin{theorem}{section}

\title{\textbf{Part 1 of Martin's Conjecture for order-preserving and measure-preserving functions}}
\author{Patrick Lutz and Benjamin Siskind}
\date{}

\begin{document}

\clearpage\maketitle
\thispagestyle{empty}

\begin{abstract}
Martin's Conjecture is a proposed classification of the definable functions on the Turing degrees. It is usually divided into two parts, the first of which classifies functions which are \textit{not} above the identity and the second of which classifies functions which are above the identity. Slaman and Steel proved the second part of the conjecture for Borel functions which are order-preserving (i.e.\ which preserve Turing reducibility). We prove the first part of the conjecture for all order-preserving functions. We do this by introducing a class of functions on the Turing degrees which we call ``measure-preserving'' and proving that part 1 of Martin's Conjecture holds for all measure-preserving functions and also that all non-trivial order-preserving functions are measure-preserving. Our result on measure-preserving functions has several other consequences for Martin's Conjecture, including an equivalence between part 1 of the conjecture and a statement about the structure of the Rudin-Keisler order on ultrafilters on the Turing degrees.
\end{abstract}

\tableofcontents

\section{Introduction}

Martin's Conjecture is a proposed classification of the definable functions on the Turing degrees, very roughly stating that every such function is either eventually constant, eventually equal to the identity function, or eventually a transfinite iterate of the Turing jump (see~\cite{montalban2019martins} for a survey). It is traditionally divided into two parts. The first states that every function is eventually constant or eventually above the identity; the second states that every function which is eventually above the identity is eventually equal to some transfinite iterate of the jump.

The conjecture was introduced by Martin in the 1970s. It remains open, but several special cases have been proved by Martin, Lachlan \cite{lachlan1975uniform}, Steel \cite{steel1982classification}, and Slaman and Steel \cite{slaman1988definable}. In particular, Slaman and Steel proved that part 2 of the conjecture holds when restricted to Borel functions which are ``order-preserving'' (i.e.\ which preserve Turing reducibility).

In this paper, we will prove that part 1 of the conjecture holds when restricted to order-preserving functions. When combined with Slaman and Steel's result, this almost completes the proof of Martin's Conjecture restricted to order-preserving functions.

We will also prove that part 1 of the conjecture holds when restricted to a class of functions which we call ``measure-preserving''. This class of functions has been implicitly considered by Martin, but, to the best of our knowledge, has not been explicitly identified before. A central thesis of this paper is that this is a natural class of functions and that studying it provides useful insight into Martin's Conjecture.

We will give two lines of evidence for this thesis. First, the class of measure-preserving functions has a few different equivalent characterizations in terms of concepts related to Martin's Conjecture. In particular, there is an ultrafilter on the Turing degrees known as the Martin measure, which is closely related to Martin's Conjecture and measure-preserving functions are exactly those functions which are measure-preserving for the Martin measure in the sense of ergodic theory. We will discuss this more thoroughly in section~\ref{sec:ultra_def}.

Second, that part 1 of Martin's Conjecture holds for measure-preserving functions has several interesting consequences.
\begin{itemize}
\item We will show that every order-preserving function is either constant on a cone or measure-preserving, so it implies part 1 of Martin's Conjecture for order-preserving functions (see section~\ref{sec:op_mp}).
\item It implies a special case of part 2 of Martin's Conjecture (see section~\ref{sec:mp_application}).
\item It implies that part 1 of Martin's Conjecture is equivalent to a statement about the structure of ultrafilters on the Turing degrees (see section~\ref{sec:ultra_rk}).
\end{itemize}

We will also show that the proof is quite general: it works in other degree structures (for example, for the arithmetic degrees and the hyperarithmetic degrees), for functions on $\Cantor$ which are not required to be Turing-invariant (and thus do not induce a function on the Turing degrees) and for functions which take values in the set of all Turing ideals.

For the rest of this introduction, we will explain the statement of Martin's Conjecture and mention some past work on it, give a definition of the class of measure-preserving functions on the Turing degrees, and provide background material necessary for some of our proofs.

\subsubsection*{Acknowledgements}

Thanks to John Steel, Vittorio Bard, James Walsh, Andrew Marks, Takayuki Kihara, Rapha\"el Carroy, William Chan and especially Ted Slaman and Gabe Goldberg for useful conversations, suggestions and references.

\subsection{Statement of Martin's Conjecture}
\label{sec:intro_conjecture}
Before we can give the formal statement of Martin's Conjecture, there are a few things we need to explain. First, a caveat: for technical reasons, the conjecture is usually stated in terms of Turing-invariant functions on $\Cantor$ rather than functions on the Turing degrees. Second, we must explain what it means for two Turing-invariant functions to be ``eventually equal'' or for one to be ``eventually above'' the other. Third, the conjecture is false in $\ZFC$ and is usually instead stated in the theory $\ZF + \AD + \DC_\R$, which we will briefly introduce. 

\subsubsection*{Turing-invariant functions}

A function $f\colon \Cantor \to \Cantor$ is \term{Turing-invariant} if for all $x, y \in \Cantor$,
\[
  x \equiv_T y \implies f(x) \equiv_T f(y).
\]
The point is that any Turing-invariant function induces a function on the Turing degrees, but functions on the reals are easier to analyze from a descriptive set theoretic point of view.

\subsubsection*{Eventual equality}

When we say that one Turing-invariant function is ``eventually equal'' to another we mean that they are Turing equivalent on a cone of Turing degrees and when we say that one Turing-invariant function is ``eventually below'' another we mean that the first is Turing reducible to the second on a cone. Here, a \term{cone of Turing degrees} (also sometimes just called a \term{cone}) is a set of the form
\[
  \Cone(a) = \{x \in \Cantor \mid x \geq_T a\}
\]
for some $a \in \Cantor$. Such a set is also called the \term{cone above $a$} and $a$ is called the \term{base of the cone}.

More formally, for Turing-invariant functions $f, g \colon \Cantor \to \Cantor$, $f$ is \term{equal to $g$ on a cone} if for all $x$ in some cone, $f(x) \equiv_T g(x)$ (note that there is a slight abuse of terminology here since $f$ and $g$ are not literally equal on a cone, but merely Turing equivalent on a cone). Likewise, $f$ is \term{below $g$ on a cone} if for all $x$ in some cone, $f(x) \leq_T g(x)$. We will also say $f$ is \term{constant on a cone} if it is equal to a constant function on a cone.

If $f$ is equal to $g$ on a cone then we will write $f\equiv_M g$ and say that they are \term{Martin equivalent}. Likewise, if $f$ is below $g$ on a cone, we will write $f \leq_M g$ and say $f$ is \term{Martin below} $g$. Note that $\leq_M$ forms a quasi-order on Turing-invariant functions, sometimes called the \term{Martin order}.

\subsubsection*{The Axiom of Determinacy}

We have said that Martin's Conjecture is stated in the theory $\ZF + \AD + \DC_\R$. Here, $\DC_\R$ denotes the Axiom of Dependent Choice on $\Cantor$ and $\AD$ denotes the Axiom of Determinacy. The Axiom of Determinacy is a strong axiom of set theory which is inconsistent with the Axiom of Choice but equiconsistent with a certain large cardinal principle (the existence of infinitely many Woodin cardinals~\cite{koellner2010large}). 

One reason for stating Martin's Conjecture under $\AD$ is that restricted versions of the axiom are true for various classes of definable sets under much weaker hypotheses. For example, Martin proved that a version of $\AD$ for Borel sets (known as Borel Determinacy) is provable in $\ZF$ and a version for $\bm{\Pi}^1_1$ sets is provable from the existence of a measurable cardinal~\cite{martin1985purely, martin1969measurable}. Thus one might hope that a proof of Martin's Conjecture under $\AD$ would yield a $\ZF$ proof of Martin's Conjecture restricted to Borel functions and a proof for analytic functions assuming the existence of a measurable cardinal.

Another key reason to use determinacy is that we have the following theorem, which makes it at least plausible that we might be able to classify Turing-invariant functions by their behavior on a cone.

\begin{theorem}[$\ZF + \AD$; Martin's Cone Theorem~\cite{martin1968axiom}]
Every set of Turing degrees either contains a cone or is disjoint from a cone.  
\end{theorem}

It is often useful to restate this theorem in a different form.

\begin{definition}
A set $A \subseteq \Cantor$ is \term{cofinal in the Turing degrees} (also sometimes just \term{cofinal}) if for every $a \in \Cantor$ there is some $x \in A$ such that $a \leq_T x$.
\end{definition}

Note that a set of Turing degrees is cofinal if and only if its complement does \textit{not} contain a cone. Thus Martin's Cone Theorem is equivalent to the statement that every cofinal set of Turing degrees contains a cone. This form is useful because it means that to prove that some property holds on a cone, it is enough to prove that it holds cofinally.

\subsubsection*{Formal statement of Martin's Conjecture}

We can now give the formal statement of Martin's Conjecture.

\begin{conjecture*}[Martin's Conjecture]
Assuming $\ZF + \AD + \DC_\R$, both of the following hold:
\begin{enumerate}
\item Every Turing-invariant function $f \colon \Cantor \to \Cantor$ is either constant on a cone or above the identity function on a cone.
\item The Martin order restricted to Turing-invariant functions which are above the identity on a cone is a prewellorder in which the successor of any function $f$ is the jump of $f$ (i.e.\ the function $x \mapsto f(x)'$).
\end{enumerate}
\end{conjecture*}

The second part of the conjecture can be interpreted as stating that every Turing-invariant function which is above the identity on a cone is equal to some transfinite iterate of the Turing jump on a cone, the idea being that a function with ordinal rank $\alpha$ in the Martin order is the $\alpha^{\text{th}}$ iterate of the Turing jump.

\subsection{Prior work on Martin's Conjecture}
\label{sec:intro_prior}
We will now state a few special cases of Martin's Conjecture which are already known. First we must state some more definitions.

\begin{definition}
A Turing-invariant function $f\colon \Cantor \to \Cantor$ is:
\begin{itemize}
\item \term{regressive} if for all $x$, $f(x) \leq_T x$ (i.e.\ $f$ is below the identity).
\item \term{order-preserving} if for all $x, y \in \Cantor$
  \[
    x \leq_T y \implies f(x) \leq_T f(y).
  \]
\item \term{uniformly invariant} (or uniformly Turing-invariant) if there is a function $u \colon \N^2 \to \N^2$ such that for all $x, y \in \Cantor$, if $i$ and $j$ are indices for Turing functionals witnessing that $x \equiv_T y$---i.e.\ $\Phi_i(x) = y$ and $\Phi_j(y) = x$---then $u(i, j)$ is a pair of indices for Turing functionals witnessing that $f(x) \equiv_T f(y)$.
\end{itemize}
\end{definition}

\begin{theorem}[$\ZF + \AD$; Slaman and Steel~\cite{slaman1988definable}]
Martin's Conjecture holds for all regressive functions---i.e.\ if $f\colon \Cantor \to \Cantor$ is a regressive Turing-invariant function then either $f$ is constant on a cone or $f$ is above the identity on a cone.
\end{theorem}

\begin{theorem}[Slaman and Steel~\cite{slaman1988definable}]
Part 2 of Martin's Conjecture holds for all Borel order-preserving functions $f\colon \Cantor \to \Cantor$.
\end{theorem}

\begin{theorem}[$\ZF + \AD$; Steel~\cite{steel1982classification}, Slaman and Steel~\cite{slaman1988definable}]
Martin's Conjecture holds for all uniformly Turing-invariant functions.
\end{theorem}

As we have mentioned, we will prove that part 1 of Martin's Conjecture holds for all order-preserving functions, complementing Slaman and Steel's result above. We will also prove that part 1 of Martin's Conjecture holds for all measure-preserving functions, a class of functions which we will now define.

\subsection{Measure-preserving functions}
\label{sec:intro_mp}
A measure-preserving function is a function on the Turing degrees which eventually gets above every fixed degree (which you might also think of as a function which ``goes to infinity in the limit''). This is made precise in the following definition.

\begin{definition}
\label{def:intro_mp}
A Turing-invariant function $f \colon 2^\omega \to 2^\omega$ is \term{measure-preserving} if for every $a \in 2^\omega$, there is some $b \in 2^\omega$ such that
\[
x \ge_T b \implies f(x) \ge_T a.
\]
In other words, for every $a$, $f$ is above $a$ on a cone.
\end{definition}

One of the earliest results on Martin's Conjecture is a proof by Martin that Martin's Conjecture holds for regressive measure-preserving functions. We will give this proof in section~\ref{sec:framework_pointed} as an example of some of the techniques we will use in our proof of part 1 of Martin's Conjecture for measure-preserving functions.

Note that any function which is above the identity on a cone is measure-preserving. Thus, restricting to the class of measure-preserving functions does not change the statement of part 2 of Martin's Conjecture.

\subsubsection*{Measure-preserving functions and the Martin order}

It is also possible to define the class of measure-preserving functions in terms of the Martin order. We omit the simple proof.

\begin{proposition}
A Turing-invariant function $f \colon 2^\omega \to 2^\omega$ is measure-preserving if and only if $f$ is Martin above every constant function.
\end{proposition}

This characterization of measure-preserving functions allows us to fit our result on part 1 of Martin's Conjecture for measure-preserving functions into a developing picture of what the Martin order looks like under $\AD$.

It is relatively easy to use determinacy to show that if a Turing-invariant function is Martin below a constant function then it must be constant on a cone. Thus the constant functions form an initial segment of the Martin order isomorphic to the partial order of the Turing degrees.

This initial segment has a natural upper bound, the identity function. Martin's result on regressive measure-preserving functions shows that it is a minimal upper bound: any function which is below the identity but above every constant function must be equivalent to the identity. Our result on part 1 of Martin's Conjecture for measure-preserving functions shows that it is actually a \textit{least} upper bound: any function which is an upper bound for all the constant functions must be above the identity. Furthermore, Slaman and Steel's result on regressive functions shows that it is not above any non-constant function. 

Thus our picture of the Martin order under $\AD$ is as follows: there is an initial segment isomorphic to the Turing degrees, consisting of the constant functions. This initial segment has a least upper bound, the identity function, which additionally is not upper bound for any non-constant function. The remaining case of part 1 of Martin's Conjecture is to rule out functions which are ``off to the side'' of the constant functions in the Martin order---for example, functions which are incomparable to all nonzero constant functions. This is illustrated in Figure~\ref{fig:martinorder}.

\begin{figure}
\centering
\begin{tikzpicture}[scale=0.75]
\node[inner sep=0, outer sep=0] (0) at (0, -3.5) {};
\node[inner sep=0] (l) at (-1.5, -1.2) {};
\node[inner sep=0] (r) at (1.5, -1.2) {};
\node[inner sep=0] (l1) at (-1, 1.1) {};
\node[inner sep=0] (r1) at (1, 1.1) {};
\node[inner sep=0] (l2) at (-3, 1.1) {};
\node[inner sep=0] (r2) at (3, 1.1) {};
\node (dt) at (0, -2) {$\D_T$};
\node (below_id) at (0, -0.5) {};
\node (below_id_label) at (-5, -0.5) {Empty (Martin)};
\node (above_id) at (0, 0.7) {};
\node (above_id_label) at (-6, 0.7) {Part 2 of Martin's Conjecture};
\node (left_id) at (-1.3, 0.1) {};
\node (left_id_label) at (-5.3, 0.1) {Empty (this paper)};
\node (right_id) at (1.3, -0.5) {};

\node[fill=black, circle, inner sep=1.5pt] (id) at (0, 0) {};
\node[right] at (id.east) {\quad$[\id]_{\equiv_M}$};
\node[draw, ellipse, minimum width=85pt, minimum height=30pt] (part1) at (-3, -2.5) {Unknown};

\draw (l) -- (id) -- (r);
\draw (0) -- (r2);
\draw (0) -- (l2);
\draw (id) -- (l1);
\draw (id) -- (r1);
\draw[dashed] (l) to[out=25,in=155] (r);
\draw [->, shorten <=4pt] (above_id_label) -- (above_id);
\draw [->, shorten <=4pt] (below_id_label) -- (below_id);
\draw [->, shorten <=4pt] (left_id_label) -- (left_id);
\end{tikzpicture}
\caption{A picture of what's known about the Martin order.}
\label{fig:martinorder}
\end{figure}
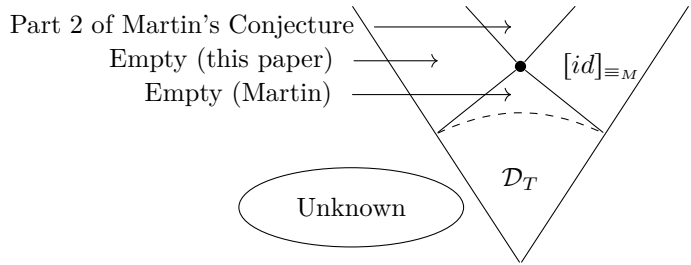

\subsection{Technical Preliminaries}
\label{sec:intro_prelim}
We will now review some technical material that we will need throughout the paper.

\subsubsection*{Pointed perfect trees}

A key technical tool in a lot of work on Martin's Conjecture is the notion of a pointed perfect tree.

\begin{definition}
A \term{pointed perfect tree} is a perfect tree $T$ such that for every $x \in [T]$, $T \leq_T x$.
\end{definition}

The key property of pointed perfect trees is that they contain a representative of every Turing degree in some cone. The idea is that if $T$ is a perfect tree then any $x \in \Cantor$ can be thought of as describing a path through $T$: at each branching point in $T$ we use the next bit of $x$ to decide whether the path should go left or right. Call the resulting path $\tilde{x}$. By construction, $T \oplus x$ can compute $\tilde{x}$. But $T\oplus \tilde{x}$ can also compute $x$ by checking whether $\tilde{x}$ goes left or right at each branching point. If $x \geq_T T$ and $T$ is pointed (so $\tilde{x} \geq_T T$), then $x \equiv_T \tilde{x}$. This is summarized by the following proposition.

\begin{proposition}
If $T$ is a pointed perfect tree, then for every $x \in \Cone(T)$ there is some $\tilde{x} \in [T]$ such that $x \equiv_T \tilde{x}$.
\end{proposition}

There is also a strengthening of Martin's cone theorem that works for arbitrary sets of reals rather than sets of Turing degrees and gives pointed perfect trees rather than cones. The proof is more or less identical to the proof of the cone theorem and is also due to Martin.

\begin{theorem}[$\ZF + \AD$; \cite{marks2016martins}]
\label{thm:intro_pointed}
If $A \subseteq \Cantor$ is cofinal in the Turing degrees then there is some pointed perfect tree $T$ such that $[T] \subseteq A$.
\end{theorem}

The lesson of this theorem is that, under $\AD$, if you want to find a pointed perfect tree whose paths all have some property, then it is enough to find a cofinal set whose elements all have that property.

\subsubsection*{Variants of the Axiom of Choice and the Axiom of Determinacy}

The Axiom of Determinacy is inconsistent with the Axiom of Choice, but it is consistent with several weak forms of the Axiom of Choice. We will need to use a few of these so we will review them here. The following axioms are listed in order of increasing logical strength.
\begin{itemize}
\item \textbf{The Axiom of Countable Choice for reals, $\CC_\R$:} This axiom states that every countable collection $\{A_n\}_{n \in \omega}$ of nonempty subsets of $\Cantor$ has a choice function. This is implied by $\AD$.
\item \textbf{The Axiom of Dependent Choice for reals, $\DC_\R$:} This axiom states that if $R$ is a binary relation on $\Cantor$ such that for every $x \in \Cantor$ there is some $y \in \Cantor$ for which $R(x, y)$ holds then there is some countable sequence of reals $\{a_n\}_{n \in \omega}$ such that for each $n$, $R(a_n, a_{n + 1})$ holds. Whether this is provable in $\ZF + \AD$ is open, but $\ZF + \AD + \DC_\R$ is equiconsistent with $\ZF + \AD$.
\item \textbf{Uniformization for sets of reals, $\Unif_\R$:} This axiom states that if $R$ is a binary relation on $\Cantor$ such that for each $x$ there is some $y$ for which $R(x, y)$ holds then $R$ can be uniformized---i.e.\ there is some function $f\colon \Cantor \to \Cantor$ such that for each $x$, $R(x, f(x))$ holds. This is \textit{not} provable in $\ZF + \AD$ and $\ZF + \AD + \Unif_\R$ is not equiconsistent with $\ZF + \AD$, but the consistency of $\ZF + \AD + \Unif_\R$ is provable from sufficiently strong large cardinal principles.
\end{itemize}

Martin's Conjecture is typically stated in the theory $\ZF + \AD + \DC_\R$ but many results related to Martin's Conjecture only need $\CC_\R$, not $\DC_\R$ (and thus are provable in $\ZF + \AD$). In this paper we will sometimes use $\Unif_\R$ and thus some of our results are proved in the theory $\ZF + \AD + \Unif_\R$. For some of these results, the proof can be made to work with $\DC_\R$ instead of $\Unif_\R$ but for a few, this is not apparent (in particular, our results on ultrafilters on the Turing degrees discussed in section~\ref{sec:ultra}). 

Throughout the paper, we will also occasionally refer to the theory $\ZF + \AD^+$. This theory is a strengthening of $\ZF + \AD$ due to Woodin which does not imply $\Unif_\R$ but which implies many consequences of $\AD + \Unif_\R$ (including $\DC_\R$). In particular, all of the results in this paper which are proved in the theory $\ZF + \AD + \Unif_\R$ can also be proved in the theory $\ZF + \AD^+$. This may seem like an obscure technical point, but it is significant for the following reason. Unlike $\AD + \Unif_\R$, it follows from sufficiently strong large cardinal principles that $\AD^+$ holds in $L(\R)$. Thus, assuming those same large cardinal principles, any instance of Martin's Conjecture that holds under $\AD^+$ also holds for all functions in $L(\R)$ (which constitute a very generous notion of the class of ``definable functions'').

\subsection{Notation and conventions}
\label{sec:intro_notation}
A number of times in this paper we will need to go back and forth between a real and its Turing degree or a Turing-invariant function on the reals and the function on the Turing degrees that it induces. To help make such transitions clearer, we will use the following notation.
\begin{itemize}
\item $\D_T$ denotes the Turing degrees, $\Cantor/\equiv_T$.
\item For $x \in \Cantor$, $\degreeof(x)$ denotes the Turing degree of $x$.
\item Lightface letters denotes reals and boldface letters denote Turing degrees. E.g.\ $a, b, x, y$ refer to elements of $\Cantor$ and $\degree{a}, \degree{b}, \degree{x}, \degree{y}$ refer to elements of $\D_T$.
\item If a lower case letter denotes a Turing-invariant function from $\Cantor$ to $\Cantor$, then the corresponidng upper case letter denotes the function on the Turing degree it induces. E.g.\ if $f \colon \Cantor \to \Cantor$ is a Turing-invariant function then $F \colon \D_T \to \D_T$ denotes the function $F(\degreeof(x)) = \degreeof(f(x))$.
\item $\id$ denotes the identity function on $\Cantor$ and $j$ denotes the Turing jump as a function on $\Cantor$.
\end{itemize}
Sometimes our slippage between Turing-invariant functions and functions on the Turing degrees will cause abuses of terminology. For example, we will often say that a Turing-invariant function is constant on a cone when it is really the induced function on the Turing degrees that is constant on a cone, and we will say that two Turing-invariant functions are equal on a cone when they are really just Turing equivalent on a cone.

We will also use the following other conventions.
\begin{itemize}
\item Unless explicitly stated otherwise, all results hold in $\ZF$ and all results which are proved for all functions in $\ZF + \AD$ hold for all Borel functions in $\ZF$.
\item A \term{Turing functional} is a program with an oracle. If $\Phi$ is a Turing functional and $x \in \Cantor$ then $\Phi(x)$ denotes the element of $\Cantor$ computed by $\Phi$ when using $x$ as an oracle and $\Phi(x, n)$ denotes the output of $\Phi$ when using $x$ as an oracle and when given input $n$ (so $\Phi(x) = n \mapsto \Phi(x, n)$).
\item We will think of a Turing functional $\Phi$ as a partial function on $\Cantor$ defined by $x \mapsto \Phi(x)$, where $x$ is in the domain of the function whenever $\Phi(x)$ is total.
\item We will assume we have a fixed computable enumeration $\Phi_0, \Phi_1, \Phi_2,\ldots$ of Turing functionals.
\item We will use $\Phi(x, n)[m]$ to denote the program $\Phi$ run with oracle $x$ on input $n$ for up to $m$ steps and $\Phi(\sigma)$ (where $\sigma \in 2^{< \omega}$) to denote the result of running $\Phi$ and using $\sigma$ to answer oracle queries (and when $\Phi$ asks a question about the oracle past the length of $\sigma$, the program diverges).
\item If $T$ is a tree and $\sigma$ is a node in $T$ then $T_\sigma$ is the tree consisting of all nodes in $T$ which are compatible with $\sigma$, i.e. $T_\sigma=\{\tau\in T\mid \tau\subseteq \sigma \text{ or } \sigma\subseteq \tau\}$.
\end{itemize}

\section{How to Prove Instances of Part 1 of Martin's Conjecture}
\label{sec:framework}

In this section we will describe a strategy which can be used to prove instances of part 1 of Martin's Conjecture. In other words, a strategy for proving that a Turing-invariant function $f\colon \Cantor\to\Cantor$ is either constant on a cone or above the identity on a cone. We will use this strategy in section~\ref{sec:mp} to prove part 1 of Martin's Conjecture for measure-preserving functions.

\subsection{The basic strategy}
\label{sec:framework_injective}
The main idea underlying our strategy is actually just the computability theory version of a basic topological fact.

\begin{quote}
\textbf{Basic topological fact:} If $f: X \to X$ is a continuous, injective function on a compact, Hausdorff space, then $f$ has a continuous inverse $f^{-1}:\range(f) \to X$.
\end{quote}

\begin{quote}
\textbf{Computability theory version:} If $f : 2^\omega \to 2^\omega$ is a computable, injective function, then for each $x$, $f(x)$ can compute $x$.
\end{quote}
The point is that if a function on $2^\omega$ is computable and injective then it is automatically above the identity. Hence one way to prove that a function $f$ is above the identity is to find a computable, injective function $g$ such that $g$ is below $f$.

In practice, it is often not possible to find such a function which is defined on all of $2^\omega$, so we will instead try to find one which is defined only on a pointed perfect tree. However, such functions are not necessarily above the identity. Instead, they satisfy the following weaker property.
\begin{quote}
If $T$ is a perfect tree and $g \colon [T] \to \Cantor$ is computable and injective then for each $x \in [T]$, $g(x) \oplus T \geq_T x$.
\end{quote}
In other words, $g$ is only above the identity after joining with a constant.

All this suggests the following strategy for proving that a Turing-invariant function $f\colon \Cantor \to \Cantor$ is above the identity on a cone:
\begin{enumerate}
\item Find a pointed perfect tree $T$ and a computable, injective function $g\colon [T] \to \Cantor$ which is below $f$. This shows that for all $x \in [T]$, $x \leq_T f(x)\oplus T$.
\item Show that for all $x$ on a cone, $f(x) \geq_T T$.
\item Put these two facts together to show that for all $x$ on a cone, $x \leq_T f(x)$.
\end{enumerate}
The third step is easy and we take care of it below. If $f$ is measure-preserving then the second step also follows immediately. Thus, in our proof of part 1 of Martin's Conjecture for measure-preserving functions, we don't need to worry about this step. In our proof of part 1 of Martin's Conjecture for order-preeserving functions we will take care of it by showing that all non-trivial order-preeserving functions are measure-preserving.

This leaves us with the question of how to carry out the first step of the strategy: how can we find $T$ and $g$ with the necessary properties? In the next two sections we will introduce some techniques which can help answer this question. But first, we will prove the statement about computable, injective functions on perfect trees mentioned above and show formally that if $f$ is measure-preserving and we can find a function $g$ with the properties listed above then $f$ is above the identity on a cone.

\begin{lemma}
\label{lemma:framework_continuousinverse}
If $T$ is a perfect tree and $g \colon [T] \to \Cantor$ is computable and injective then for each $x \in [T]$,
\[
g(x) \oplus T \ge_T x.
\]
\end{lemma}

\begin{proof}
Since $g$ is computable, there is some Turing functional $\Phi$ such that for all $x \in [T]$, $\Phi(x)$ is total and equal to $g(x)$. So it suffices to prove that for all $x \in [T]$, $\Phi(x)\oplus T \geq_T x$. The main idea of the proof is just a routine application of compactness.

First, we will give an algorithm to compute $x$ given $\Phi(x)$ and $T$. Say we want to compute $x \restriction n$. For each $\sigma$ in level $n$ of $T$ we do the following search (and we do all of these searches in parallel):
\begin{quote}
    Look for an $m > n$ such that for all descendants $\tau$ of $\sigma$ on level $m$ of $T$, $\Phi(\tau)[m]$ disagrees with $\Phi(x)\restriction m$.
\end{quote}
Once all but one of these searches have terminated, we output the remaining element of level $n$ of $T$ as our guess for $x\restriction n$.

Hopefully it is clear that this search will never terminate for $x\restriction n$ (since on every level above $m > n$ there is a descendant of $x\restriction n$ in $T$, namely $x\restriction m$, which will not make $\Phi$ disagree with $\Phi(x)$). So all we really need to do is show that the search will terminate for every $\sigma$ in level $n$ of $T$ which is not equal to $x\restriction n$.

Suppose this is not the case and let $\sigma$ be such a node in $T$. Then by K\"onig's lemma we can find some $y \in [T]$ extending $\sigma$ such that for all $m$, $\Phi(y)[m]$ does not disagree with $\Phi(x)$. However, we know that $\Phi$ is total on $y$ and injective on $[T]$, hence $\Phi(y)$ and $\Phi(x)$ must disagree somewhere, a contradiction.
\end{proof}

\begin{lemma}
\label{lemma:framework_main}
Suppose $f \colon \Cantor \to \Cantor$ is Turing-invariant and measure-preserving, $T$ is a pointed perfect tree and $g\colon [T] \to \Cantor$ is computable, injective and below $f$. Then for all $x$ on a cone, $f(x) \geq_T x$.
\end{lemma}

\begin{proof}
Since $f$ is measure-preserving, there is some cone on which $f(x)$ is always above $T$. We may also assume that this cone is high enough that all its elements are above $T$. Let $x$ be any element of this cone and let $\tilde{x}$ be an element of $[T]$ in the same Turing degree as $x$ (which must exist since $T$ is a pointed perfect tree). We can then calculate
\begin{align*}
  \tilde{x} &\leq_T g(\tilde{x}) \oplus T & \text{by the previous lemma}\\
  &\leq_T f(\tilde{x}) &\text{since $g$ is below $f$ and $f(\tilde{x}) \geq_T T$.}
\end{align*}
Since $x \equiv_T \tilde{x}$ and $f$ is Turing-invariant, this implies that $x \leq_T f(x)$.
\end{proof}

\subsection{Finding pointed perfect trees}
\label{sec:framework_pointed}
In the previous section, we outlined a general strategy to prove that a function $f\colon \Cantor \to \Cantor$ is above the identity. A key step involved finding a computable, injective function $g$ below $f$ which is defined on a pointed perfect tree. In this section we will discuss some lemmas which are useful for finding such a $g$ and then give an example of using these lemmas to prove an instance of part 1 of Martin's Conjecture.

\subsubsection*{Finding pointed perfect trees using determinacy}

Recall from the introduction the following theorem due to Martin, which is often useful for finding pointed perfect trees under $\AD$.

\begin{theorem}[$\ZF + \AD$; Martin; \cite{marks2016martins}, Lemma 3.5]
\label{lemma:framework_pointed}
Suppose $A \subseteq 2^\omega$ is cofinal in the Turing degrees. Then there is a pointed perfect tree $T$ such that $[T] \subseteq A$.
\end{theorem}

The following consequence of Theorem~\ref{lemma:framework_pointed} is also useful.

\begin{corollary}[$\ZF + \AD$]
\label{cor:framework_pointed}
Suppose $\seq{A_n}_{n \in \N}$ is a countable sequence of subsets of $\Cantor$ such that $\bigcup_n A_n$ is cofinal in the Turing degrees. Then there is some $n \in \N$ and some pointed perfect tree $T$ such that $[T] \subseteq A_n$.
\end{corollary}

\begin{proof}
It is enough to show that some $A_n$ must be cofinal. Suppose not. So for each $n$, there is some $x_n \in \Cantor$ such that $A_n$ is disjoint from the cone above $x_n$. But then any $y \geq_T \bigoplus_n x_n$ cannot be in any of the $A_n$'s, contradicting the fact that $\bigcup_n A_n$ is cofinal.
\end{proof}

There is a further very easy consequence of this corollary which has proved surprisingly useful. This consequence is not new and has often been used implicitly in research on Martin's Conjecture, but we have found it helpful to formulate it as an explicit principle.

\begin{lemma}[$\ZF + \AD$; Computable uniformization lemma]
\label{lemma:framework_computableuniformization}
Suppose $R$ is a binary relation on $2^\omega$ such that both of the following hold.
\begin{itemize}
    \item The domain of $R$ is cofinal: for all $a$ there is some $x \ge_T a$ and some $y$ such that $(x, y) \in R$
    \item $R$ is a subset of Turing reducibility: for every $(x, y) \in R$, $x \ge_T y$.
\end{itemize}
Then there is a pointed perfect tree $T$ and a computable function $f \colon [T] \to \Cantor$ such that for all $x \in [T]$, $(x, f(x)) \in R$. In other words, $f$ is a computable choice function for $R$ on $[T]$.
\end{lemma}

\begin{proof}
For each $n \in \N$, let $A_n$ be the set of $x$ such that $\Phi_n(x)$ is total and $R(x, \Phi_n(x))$ holds. For each $x$ in the domain of $R$, there must be some $n$ for which this holds and thus $\bigcup_n A_n = \dom(R)$ is cofinal. So by Corollary~\ref{cor:framework_pointed}, there is some $n$ and pointed perfect $T$ such that $[T] \subseteq A_n$. By construction, $T$ and $\Phi_n$ satisfy the conclusion of the lemma.
\end{proof}

Later, we will need the following corollary of this lemma, which also shows how it is typically used. The corollary says that any increasing function can be inverted by a computable function on a pointed perfect tree. Note that the function $f$ in the statement of the corollary is not required to be Turing-invariant.

\begin{corollary}
\label{cor:framework_invertincreasing}
If $f\colon 2^\omega \to 2^\omega$ is a function such that $f(x) \ge_T x$ for all $x$ then there is a pointed perfect tree $T$ and a computable function $g\colon [T] \to \Cantor$ which is a right inverse for $f$ on $[T]$. That is, for all $x \in [T]$, $f(g(x)) = x$.
\end{corollary}

\begin{proof}
Let $R$ be the binary function defined as follows.
\[
  R(x, y) \iff x = f(y).
\]
Applying Lemma \ref{lemma:framework_computableuniformization} to this relation gives us what we want. To show that we can apply the lemma, we need to check that $R$ is a subset of Turing reducibility and that its domain is cofinal. The former is a consequence of the fact that $f(y) \ge_T y$ for all $y$. For the latter, consider any $a \in 2^\omega$. We need to show that there is some $x \ge_T a$ which is in the domain of $R$. For this, we can just take $x = f(a)$.
\end{proof}

\subsubsection*{Refining pointed perfect trees}

The next lemma is useful for building injective functions on pointed perfect trees. It is relatively well-known but we include a proof anyway for the sake of completeness.

\begin{lemma}[Tree thinning lemma]
\label{lemma:framework_thinning}
If $T$ is a pointed perfect tree and $f$ is a computable function defined on $[T]$ then one of the following must hold:
\begin{itemize}
    \item We can ``thin out'' $T$ to make $f$ injective: there is a pointed perfect tree $S$ such that $S \subseteq T$ and $f$ is injective on $[S]$.
    \item $f$ is constant on a large set: there is a node $\sigma$ in $T$ such that $f$ is constant on $[T_\sigma]$.
\end{itemize}
In particular, $f$ is either constant or injective on a pointed perfect subtree of $T$.
\end{lemma}

\begin{proof}
The idea is basically the same as in Spector's construction of a minimal degree (i.e.\ Sacks forcing). Suppose that the second condition does not hold---i.e.\ that $f$ is not constant on $T_\sigma$ for any $\sigma$ in $T$. We will show how to find a pointed perfect tree $S \subseteq T$ on which $f$ is injective.

Let $\Phi$ be a Turing functional such that for all $x \in [T]$, $\Phi(x)$ is total and agrees with $f(x)$. We will define $S$ in a series of stages. In stage $0$, we let $S_0$ consist of just the empty sequence (i.e.\ the root node of $T$). In stage $n + 1$ we have a finite tree $S_n \subseteq T$ which we want to extend to $S_{n + 1}$ in a way that makes sure that every leaf in $S_n$ has two incompatible extensions in $S_{n + 1}$ and $\Phi$ is injective on the leaves of $S_{n + 1}$. This is actually pretty straightforward to do: for each leaf $\sigma$ of $S_n$ we know that $\Phi$ is not constant on $[T_\sigma]$ so we can find descendants $\tau_1$ and $\tau_2$ of $\sigma$ in $T$ and an $m$ such that
\[
\Phi(\tau_1)[m] \text{ disagrees with } \Phi(\tau_2)[m]
\]
(i.e.\ there is a place where they both converge and are not equal). Put these $\tau_1$ and $\tau_2$, along with all their ancestors, into $S_{n + 1}$.

Now define $S$ as the union of all the $S_n$'s. It is clear that if we construct $S$ in this way then $S$ is a perfect tree, $S \subseteq T$, and $\Phi$ (and hence $f$) is injective on $S$. To see that $S$ is pointed, just note that the above process was computable in $T$ and so $S \le_T T$. Since $[S] \subseteq [T]$ and each element of $[T]$ computes $T$, we have that each element of $[S]$ computes $T$ and hence also computes $S$.
\end{proof}

\subsubsection*{Example: Martin's Conjecture for regressive, measure-preserving functions}

We will now give an example of using the strategy outlined in the previous section, along with the lemmas above, to prove part 1 of Martin's Conjecture for some class of functions: namely, functions which are both regressive and measure-preserving.

As we mentioned in the introduction, this result was first proved by Martin in the 1970s (though he didn't use the term ``measure-preserving''). He didn't publish his proof, but a proof was included in a paper by Steel~\cite{steel1982classification}. Later, Slaman and Steel proved that a modified version of this theorem is still true even in $\ZFC$~\cite{slaman1988definable}.

\begin{theorem}[$\ZF + \AD$; Martin]
Suppose $f\colon \Cantor \to \Cantor$ is a Turing-invariant function which is both regressive and measure-preserving. Then for all $x$ on a cone, $f(x) \geq_T x$.
\end{theorem}

\begin{proof}
Since $f$ is measure-preserving, Lemma~\ref{lemma:framework_main} implies that it suffices to find a pointed perfect tree $T$ and a computable, injective function $g \colon [T] \to \Cantor$ such that for all $x \in [T]$, $g(x) \leq_T f(x)$. The idea is that since $f$ is regressive, we can just use $g = f$.

By applying the computable uniformization lemma (Lemma~\ref{lemma:framework_computableuniformization}) to $f$ itself (i.e. to the relation $\{(x, y) \mid y = f(x)\}$) we obtain a pointed perfect tree $T$ such that $f$ is computable on $[T]$. By the tree thinning lemma (Lemma~\ref{lemma:framework_thinning}), either $f$ is constant on a pointed perfect subtree of $T$ or $f$ is injective on a pointed perfect subtree of $T$. In the latter case, we are done. So it suffices to prove that $f$ cannot be constant on any pointed perfect tree.

Suppose that $f$ is constant on a pointed perfect tree. Since a pointed perfect tree contains a representative of every Turing degree on a cone, this means that $f$ is constant on a cone. But that contradicts the assumption that $f$ is measure-preserving.
\end{proof}

\subsection{Ordinal invariants}
\label{sec:framework_ordinal}
Suppose that we have a Turing-invariant function $f\colon \Cantor \to \Cantor$ and we are trying to find a computable, injective function below $f$ which is defined on a pointed perfect tree. Here's a naive way we might go about this. First, define a binary relation $R$ on $2^\omega$ by
\[
R(x, y) \iff y \leq_T x \text{ and } y \leq_T f(x).
\]
By the computable uniformization lemma applied to $R$, we get a computable function $g$ which is below $f$. We might then try to apply the tree thinning lemma to find a pointed perfect tree on which $g$ is injective. But there is one problem: if $g$ is constant on a cone (or more generally, on a pointed perfect tree) then we cannot apply that lemma.

One solution to this problem is to make the definition of $R$ more restrictive to ensure that $g$ cannot be constant on a cone. In this section, we will introduce an idea which can help do this. Briefly stated, here's the idea: identify some function $\alpha$ from Turing degrees to ordinals and modify $R$ to
\[
R(x, y) \iff y \leq_T x \text{ and } y \leq_T f(x) \text{ and }\alpha(x) = \alpha(y).
\]
Then when we apply the computable uniformization lemma to $R$, we get a function $g$ which is not only computable and below $f$, but also preserves $\alpha$. As long as $\alpha$ is not constant on any cone, neither is $g$. We will refer to such a function $\alpha$ as an \term{ordinal invariant}\footnote{Such functions are often studied in work on Martin's Conjecture and determinacy more generally, but the term ``ordinal invariant'' is not standard and our use of them is somewhat different from their usual role.}.

Of course, to make this work we need to choose $\alpha$ so that the domain of the relation $R$ defined above is cofinal (otherwise we cannot use the computable uniformization lemma). But if we can do this, then the argument sketched above is valid and we will use it in section~\ref{sec:mp_proof2}. In the remainder of this section, we will formally introduce ordinal invariants, note a few of their properties and formalize the argument sketched above.

\begin{definition}
An \term{ordinal invariant} is a Turing-invariant function $\alpha \colon 2^\omega \to \Ord$ (where $\Ord$ denotes the class of ordinals), i.e.\ a function $\alpha \colon \Cantor \to \Ord$ such that if $x \equiv_T y$ then $\alpha(x) = \alpha(y)$.
\end{definition}

\begin{example}
The quintessential example of an ordinal invariant is the function $x\mapsto \omega_1^x$, i.e.\ the function mapping a real $x$ to the least ordinal with no presentation computable from $x$.
\end{example}

At first it may seem that the notion of an ordinal invariant is much too general to be interesting, but it turns out that, assuming determinacy, it is possible to prove quite a lot about them.

For example, Martin has shown that, under $\ZF + \AD$, the function $x\mapsto \omega_1^x$ is the least nontrivial ordinal invariant: for every ordinal invariant $\alpha$, either $\alpha$ is constant on a cone, or $\alpha(x) \geq \omega_1^x$ on a cone.

Under $\AD+\DC$ (or $\AD^+$), it is easy to show that the relation ``$\alpha(x) \leq \beta(x)$ on a cone'' prewellorders the ordinal invariants. The theorem of Martin just mentioned implies that $x \mapsto \omega^x_1$ has rank $\omega_1$ in this prewellorder. Steel, in~\cite{steel1982classification}, has calculated the rank of a number of other ordinal invariants.

It is also possible to show that every ordinal invariant is order-preserving on a cone.

\begin{proposition}[$\ZF + \AD$]
If $\alpha$ is an ordinal invariant, then $\alpha$ is order-preserving on a cone---i.e.\ for all $x$ and $y$ in some cone
\[
x \le_T y \implies \alpha(x) \le \alpha(y).
\]
\end{proposition}

\begin{proof}
Define $\alpha_{\min}\colon 2^\omega \to \Ord$ by
\[
\alpha_{\min}(x) = \min \{\alpha(y) \mid y \geq_T x\}.
\]
The claim that $\alpha$ is order-preserving on a cone is equivalent to the claim that $\alpha(x) = \alpha_{\min}(x)$ on a cone. For each $x$, there is some $y\geq_T x$ such that $\alpha(y) = \alpha_{\min}(x)$ and hence $\alpha(y) =\alpha_{\min}(y)$. In other words, $\alpha(x) = \alpha_{\min}(x)$ holds cofinally. By determinacy, this means $\alpha(x) = \alpha_{\min}(x)$ on a cone.
\end{proof}

\begin{lemma}[$\ZF + \AD$]
\label{lemma:framework_ordinal}
Suppose $f\colon \Cantor \to \Cantor$ is a Turing-invariant, measure-preserving function and $\alpha \colon \Cantor \to \Ord$ is an ordinal invariant such that $\alpha$ is not constant on any cone and for cofinally many $x$, there is some $y$ such that
\begin{enumerate}
\item $y\leq_T x$,
\item $y \leq_T f(x)$, and
\item $\alpha(x) = \alpha(y)$.
\end{enumerate}
Then $f$ is above the identity on a cone.
\end{lemma}

\begin{proof}
By Lemma~\ref{lemma:framework_main}, it suffices to show that there is a pointed perfect tree $T$ and a computable, injective function $g\colon [T] \to \Cantor$ which is below $f$. We will find $g$ as described above.

First, define a binary relation $R$ by
\[
R(x, y) \iff y \leq_T x \text{ and } y \leq_T f(x) \text{ and }\alpha(x) = \alpha(y).  
\]
Note that by our assumption about $\alpha$, the domain of $R$ is cofinal. Thus there is a pointed perfect tree $T$ and a computable function $g\colon [T] \to \Cantor$ uniformizing $R$ on $[T]$. In particular, for all $x \in [T]$, $g(x) \leq_T f(x)$ and $\alpha(x) = \alpha(g(x))$.

By the tree thinning lemma, $g$ is either constant or injective on a pointed perfect subtree of $T$. If $g$ is injective on a pointed perfect tree then we are done. So it suffices to show that it is not constant on any pointed perfect subtree of $T$. Suppose it was. Then on the set of paths $x$ through this tree, $\alpha(x) = \alpha(g(x))$ would also be constant. Since any pointed perfect tree contains a representative of every Turing degree on some cone, this would contradict our assumption that $\alpha$ is not constant on any cone.
\end{proof}

\section{Part 1 of Martin's Conjecture for Measure-Preserving Functions}
\label{sec:mp}

In this section, we will prove part 1 of Martin's Conjecture for measure-preserving functions. Actually, we will give two proofs: first, a relatively straightforward proof that works in $\ZF + \AD + \Unif_\R$ and then a somewhat more complicated proof (a modification of the first proof) that works in $\ZF + \AD + \DC_\R$. Both proofs follow the basic strategy explained in section~\ref{sec:framework}; the second proof also uses the idea of ordinal invariants from section~\ref{sec:framework_ordinal}. We will finish the section by giving an application of our result to part 2 of Martin's Conjecture.

\subsection{Proof of part 1 of Martin's Conjecture for measure-preserving functions}
\label{sec:mp_proof}
We will now give our first proof of part 1 of Martin's Conjecture for measure-preserving functions.\footnote{Assuming $\AD + \Unif_\R$, though $\AD^+$ would also suffice.} The proof follows the strategy outlined in section~\ref{sec:framework}: given a measure-preserving function $f$, we will find a function defined on a pointed perfect tree which is computable, injective and below $f$. To do so, we will first associate to any measure-preserving function $f$ a family of functions called \term{increasing moduli} for $f$, which are essentially Skolem functions witnessing that $f$ is measure-preserving. We will then show:
\begin{enumerate}
\item Every measure-preserving function $f$ has an increasing modulus, $g$.
\item Every such $g$ has a computable right inverse, $h$, defined on a pointed perfect tree.
\item Every such $h$ is computable, injective and below $f$.
\end{enumerate}
Thus $h$ satisfies the properties required by Lemma~\ref{lemma:framework_main} and so $f$ is above the identity on a cone.

The second and third steps of this proof are straightforward: to find $h$, we will invoke Corollary~\ref{cor:framework_invertincreasing} on finding right inverses for increasing functions; the fact that $h$ is computable, injective and below $f$ will follow fairly directly from the definition of ``increasing modulus.'' Finding $g$, however, is trickier. We will construct $g$ using $\Unif_\R$, which is not provable in $\ZF + \AD$ (this is the only part of the proof that cannot be carried out in $\ZF + \AD$). We do not know if it is possible to prove that every measure-preserving function has a modulus in $\ZF + \AD$.

\begin{definition}
Suppose $f \colon 2^\omega \to 2^\omega$ is a measure-preserving function. A \term{modulus} for $f$ is a function $g \colon 2^\omega \to 2^\omega$ such that for all $x$ and all $y \ge_T g(x)$ we have $f(y) \ge_T x$. Note that $g$ is not required to be Turing-invariant.
\end{definition}

Here's the idea behind the definition of modulus. If $f$ is measure-preserving then for all $a \in \Cantor$, there is some $b\in\Cantor$ such that on the cone above $b$, $f(x)$ is always above $a$. A function $g$ is a modulus for $f$ if for each $a$ we can take $b = g(a)$.

For convenience, we will only use moduli which are above the identity. We will call such a modulus an \term{increasing modulus}.

\begin{definition}
If $f \colon 2^\omega \to 2^\omega$ is a measure-preserving function then a modulus $g$ for $f$ is an \term{increasing modulus} for $f$ if for all $x$, $g(x) \ge_T x$.
\end{definition}

As mentioned above, it is not clear how to show in $\ZF + \AD$ that every measure-preserving function has a modulus. However, this is easy to prove in $\ZF + \Unif_\R$.

\begin{lemma}[$\ZF + \Unif_\R$]
\label{lemma:mp_modulusexists}
If $f$ is a measure-preserving function then $f$ has an increasing modulus.
\end{lemma}

\begin{proof}
Let $R(x, y)$ be the binary relation defined by
\[
R(x, y) \iff x \le_T y \text{ and } \forall z \ge_T y\, (f(z) \ge_T x).
\]
Since $f$ is measure-preserving, we know that for each $x$, the set $\{y \mid R(x, y)\}$ is nonempty. Finding an increasing modulus for $f$ just means finding a function $g$ such that for each $x$, $R(x, g(x))$ holds---in other words, a function $g$ which uniformizes $R$. Since we are assuming $\Unif_\R$, we are done.
\end{proof}

We can now prove the main theorem of this section.

\begin{theorem}[$\ZF + \Unif_\R$]
\label{thm:mp_part1_adr}
If $f \colon 2^\omega \to 2^\omega$ is a Turing-invariant, measure-preserving function then $f$ is above the identity on a cone.
\end{theorem}

\begin{proof}
Suppose $f$ is a measure-preserving function. By Lemma~\ref{lemma:mp_modulusexists} we can find an increasing modulus $g$ for $f$. By Corollary~\ref{cor:framework_invertincreasing} we can invert $g$ on a pointed perfect tree---that is, there is a pointed perfect tree $T$ and a computable function $h$ defined on $[T]$ such that for each $x \in [T]$, $g(h(x)) = x$. Now let's review what we know about this $h$.
\begin{itemize}
\item Since $h$ is a right inverse of $g$ on $[T]$, $h$ is injective on $[T]$.
\item $h$ is below $f$ on $[T]$: if $x \in [T]$ then since $g$ is a modulus for $f$, $f(g(h(x)))$ computes $h(x)$. And since $g(h(x)) = x$, this just means that $f(x)$ computes $h(x)$.
\end{itemize}
Thus by Lemma~\ref{lemma:framework_main}, $f(x) \geq_T x$ on a cone.  
\end{proof}

\subsection{An alternate proof that works in $\AD + \DC_\R$}
\label{sec:mp_proof2}
In the previous section, we saw how to prove part 1 of Martin's Conjecture for measure-preserving functions in $\ZF + \AD + \Unif_\R$. In this section we will see how to modify the proof so that it works in the weaker theory $\ZF + \AD + \DC_\R$.\footnote{Though note that we really do use $\DC_\R$ in the proof, rather than just $\CC_\R$, in contrast to most proofs of statements related to Martin's Conjecture.} Recall that the reason we needed $\Unif_\R$ in the previous section was to show that every measure-preserving function has a modulus. In this section we will get around that difficulty by using an ordinal invariant (see section~\ref{sec:framework_ordinal}) to approximate the role of the modulus in the previous proof.

We will start by defining something called a ``modulus sequence,'' which can be thought of as a countable fragment of a modulus function for $f$.

\begin{definition}
Suppose $f\colon 2^\omega \to 2^\omega$ is a measure-preserving function and $x \in 2^\omega$. A \term{modulus sequence} for $x$ is a sequence of reals $x = x_0 \le_T x_1 \le_T x_2 \le_T \ldots$ which is increasing in the Turing degrees and such that for all $n \in \N$ and all $y \in 2^\omega$,
\[
  y \ge_T x_{n + 1} \implies f(y) \ge_T x_n.
\]
In other words, $x_1$ is large enough that $f$ is above $x$ on the cone above $x_1$, $x_2$ is large enough that $f$ is above $x_1$ on the cone above $x_2$, and so on.
\end{definition}

The idea is that if $g$ is an increasing modulus for $f$ then $x, g(x), g(g(x)), \ldots$ is a modulus sequence for $x$, but that even when $f$ does not have a modulus, it still has modulus sequences. It is easy to see that the amount of choice required to prove that modulus sequences exist is much weaker than the amount of choice that seems to be required to prove that modulus functions exist. This is expressed by the following lemma.

\begin{lemma}[$\ZF + \DC_\R$]
Suppose $f \colon 2^\omega \to 2^\omega$ is a Turing-invariant function which is measure-preserving. Then for all $x \in 2^\omega$, there is a modulus sequence for $x$.
\end{lemma}

\begin{proof}
Let $x$ be an arbitrary real. Note that a sequence $x_0, x_1, x_2, \ldots$ is a modulus sequence for $x$ as long as $x_0 = x$ and for each $n$, $x_{n + 1}$ satisfies a certain condition with respect to $x_n$. It is easy to see that since $f$ is measure-preserving, no matter what $x_n$ we have picked, there is some $x_{n + 1}$ which satisfies this condition with respect to it. Thus we can use $\DC_\R$ to pick a modulus sequence for $x$.
\end{proof}

We can now prove the theorem. Before we actually give the proof, let's briefly review how ordinal invariants can be used to carry out the general strategy from section \ref{sec:framework}. Recall that to prove that a measure-preserving function $f$ is above the identity, it is enough to find a computable function $g$ which is below $f$ and which can be made injective on a pointed perfect tree. It is easy to use the computable uniformization theorem to find functions $g$ which are computable and below $f$. But it is hard to ensure that they are not just constant on a cone (and thus cannot be injective on any pointed perfect tree). However, if we can find an ordinal invariant, $\alpha$, and a function $g$ such that $\alpha(x) = \alpha(g(x))$ for all $x$ then as long as $\alpha$ is not constant on a cone, $g$ cannot be constant on a cone.

Thus the goal of the proof below is to come up with an ordinal invariant $\alpha$ which is not constant on a cone and for which we can find a computable function $g$ which is both below $f$ and preserves $\alpha$. To do so, we will use Lemma~\ref{lemma:framework_ordinal}, which gives a list of conditions on an ordinal invariant $\alpha$ which are sufficient to guarantee that we can find such a $g$.

\begin{theorem}[$\ZF + \AD + \DC_\R$]
\label{thm:mp_part1measurepreserving}
If $f \colon 2^\omega \to 2^\omega$ is a Turing-invariant, measure-preserving function then $f$ is above the identity on a cone.
\end{theorem}

\begin{proof}
First we will define our ordinal invariant. Let $\alpha \colon 2^\omega \to \Ord$ be the function defined by
\[
  \alpha(x) = \min\{\sup\nolimits_{n \in \N}\omega_1^{x_n} \mid \langle x_n\rangle_n \text{ is a modulus sequence for $x$}\}.
\]
Observe that $\alpha(x)$ is always a countable ordinal which is at least $\omega_1^x$ and therefore $\alpha(x)$ is not constant on any cone.

By Lemma~\ref{lemma:framework_ordinal}, it suffices to check that the following set $A$ is cofinal:
\[
  A = \{x \mid \exists y\, (y \le_T x \text{ and } y \le_T f(x) \text{ and } \alpha(y) = \alpha(x))\}.
\]
So let $x$ be an arbitrary real and we will find an element of $A$ which computes $x$. To do so, let $x = x_0 \le_T x_1 \le_T x_2 \le_T \ldots$ be a modulus sequence for $x$ which witnesses the value of $\alpha(x)$ (i.e.\ such that $\alpha(x) = \sup_n \omega_1^{x_n}$).

We now claim that $x_1$ is in $A$, as witnessed by $x$. By the definition of modulus sequence, it is clear that $x \le_T x_1$ and $x \le_T f(x_1)$. So we just need to show that $\alpha(x) = \alpha(x_1)$. First observe that $\alpha(x_1)$ cannot be larger than $\alpha(x)$ because $x_1, x_2, x_3, \ldots$ is a modulus sequence for $x_1$ and so
\[
\alpha(x_1) \le \sup \{\omega_1^{x_1}, \omega_1^{x_2}, \ldots\} = \sup \{\omega_1^x, \omega_1^{x_1}, \omega_1^{x_2},\ldots\} = \alpha(x).
\]
Next observe that $\alpha(x_1)$ also cannot be smaller than $\alpha(x)$ because if $x_1 = y_0 \le_T y_1 \le_T y_2 \le_T \ldots$ is a modulus sequence for $x_1$ witnessing the value of $\alpha(x_1)$ then $x, y_0, y_1, y_2,\ldots$ is a modulus sequence for $x$ and so
\[
\alpha(x) \le \sup\{\omega_1^x, \omega_1^{y_0}, \omega_1^{y_1},\ldots\} = \sup\{\omega_1^{y_0}, \omega_1^{y_1},\omega_1^{y_2},\ldots\} = \alpha(x_1). \qedhere
\]
\end{proof}

\subsubsection*{What about Borel Functions?}

In the previous section, we saw how to prove part 1 of Martin's Conjecture for measure-preserving functions using $\AD + \Unif_\R$. A careful examination of that proof shows that when we restrict to Borel functions, the proof requires $\mathbf{\Pi}^1_1$ determinacy. And the reason that the proof requires $\mathbf{\Pi}^1_1$ rather than Borel determinacy is similar to the reason that the proof for all functions required something more than just $\AD$.

In particular, to prove that every measure-preserving function $f$ has a modulus, we had to uniformize the following relation:
\[
R(x, y) \iff x \le_T y \text{ and } \forall z \ge_T y\, (f(z) \ge_T x).
\]
When $f$ is Borel, this relation is $\mathbf{\Pi}^1_1$ and so the Kond\^o-Addison theorem says that it has a $\mathbf{\Pi}^1_1$ uniformization (but not necessarily a Borel uniformization). Thus every Borel measure-preserving function has a $\mathbf{\Pi}^1_1$ modulus. Since the rest of the proof needs to apply determinacy to sets defined using the modulus, the fact that the modulus is only guaranteed to be $\mathbf{\Pi}^1_1$ rather than Borel causes the proof to require $\mathbf{\Pi}^1_1$ determinacy rather than Borel determinacy.

Above, we saw how to prove part 1 of Martin's Conjecture for measure-preserving functions using just $\AD + \DC_\R$ rather than $\AD + \Unif_\R$. In light of the above discussion, it is reasonable to ask whether this yields a proof for Borel functions that only requires Borel determinacy (and thus works in $\ZF$). Somewhat surprisingly, the answer seems to be ``no.'' That is, the proof in this section seems to use more than Borel determinacy even when the functions considered are Borel. This is because the definition of the ordinal invariant that we use in the proof is rather complicated and so the set that we need to apply determinacy to in the proof is not Borel.

This is interesting in part because it somewhat contradicts the idea that proofs of Martin's Conjecture should only use determinacy in a ``local'' way (that is, the proof for Borel functions should only require Borel determinacy, and so on). It would be interesting to know whether part 1 of Martin's Conjecture for measure-preserving Borel functions can be proved using just Borel determinacy.

\subsection{Application to part 2 of Martin's Conjecture}
\label{sec:mp_application}
In this section we will show that our result on part 1 of Martin's Conjecture for measure-preserving functions can be applied to obtain a new result about part 2 of Martin's Conjecture.

Recall that part 2 of Martin's Conjecture says that the Martin order is a prewellorder on Turing-invariant functions which are above the identity and that the successor in this prewellorder is given by the Turing jump. As a first step towards proving this, we could try to show that all functions above the identity are comparable in the Martin order. This would not show that the quotient of the Martin order by Martin equivalence is a well order, but it would at least show it's a linear order. We will not even show this, but we will show that if we have two Turing-invariant functions $f$ and $g$ which are above the identity and which satisfy an additional assumption then they are comparable in the Martin order.

To understand the idea of the proof, consider the following way that one might try to show any two functions above the identity are comparable. Suppose we are given Turing-invariant functions $f$ and $g$ and want to show $f \le_M g$. Naively, we might try to ``subtract'' $f$ from $g$ and show that the resulting function is above the identity. That is, we might try to define a function $h$ as follows. Given a real $x$, we first try to find a $y$ such that $f(y) = x$ and then set $h(x) = g(y)$. This $h$ can be thought of as the ``difference'' between $g$ and $f$ because for the $y$ used to define $h(x)$, we have $h(f(y)) = g(y)$. If we could show that $h$ is above the identity then it would be some indication that $g$ is above $f$ since their ``difference'' is ``positive.''

There are a number of obvious problems with this strategy. First, there is no reason to expect that any $h$ we define this way will be Turing-invariant, much less amenable to known techniques for proving functions are above the identity. And second, even if $h$ is a Turing-invariant function above the identity, it is not clear that this really implies that $g$ is above $f$ on a cone (for example, the $y$'s we use to define $h$ could all lie outside of some cone). The key insight of the proof below is that if $f$ and $g$ satisfy a certain additional condition then all these problems disappear.

\begin{theorem}[$\ZF + \AD + \DC_\R$]
Suppose $f, g\colon 2^\omega \to 2^\omega$ are Turing-invariant functions which are above the identity and such that for all $x, y \in 2^\omega$,
\[
  f(x) \equiv_T f(y) \implies g(x) \equiv_T g(y).
\]
Then $f \le_M g$.
\end{theorem}

\begin{proof}
We want to define a function $h$ in the following way: given any real $x$, find some real $y$ such that $f(y) \equiv_T x$ and then define $h(x) = g(y)$. However, it may not be immediately obvious how to formally define $h$. We can do so as follows.

First, note that since $f$ is above the identity, we can apply Corollary~\ref{cor:framework_invertincreasing} to find a pointed perfect tree, $T$, and a (possibly non Turing-invariant) right inverse $k$ for $f$ defined on $[T]$. Next, recall that in section~\ref{sec:intro_prelim} we described a function $x \mapsto \tilde{x}$ that takes any $x \in \Cone(T)$ to a Turing equivalent element of $[T]$. We can then formally define $h \colon \Cone(T) \to \Cantor$ by
\[
  h(x) = g(k(\tilde{x})).
\]
However, it is probably easier to think of $h$ as defined by the informal procedure described above. Figure~\ref{fig:mp_application1} below shows how $h$ is defined.

\begin{figure}[h!]
\centering
\begin{tikzpicture}
\node[fill=black, circle, inner sep=1.5pt] (x) at (0, 0) {};
\node[above left] at (x.west) {$x$};
\node[fill=black, circle, inner sep=1.5pt] (y) at (0, -1.5) {};
\node[below left] at (y.west) {$y$};
\node[fill=black, circle, inner sep=1.5pt] (hx) at (1.4, -.35) {};
\node[right] at (hx.east) {$h(x) = g(y)$};

\draw[dashed, ->, shorten >=2pt] (x) to[out=-135,in=135] node[left] {$k$} (y);
\draw[->, shorten >=2pt] (y) to[out=45,in=-45] node[right] {$f$} (x);
\draw[->, shorten >=2pt] (y) to[out=20,in=-115] node[below right] {$g$} (hx);
\draw[->, shorten >=2pt] (x) -- node[above] {$h$} (hx);
\end{tikzpicture}
\caption{$h(x)$ is defined by finding some $y$ such that $f(y) = x$ and then setting $h(x) = g(y)$. The real $y$ can be found using $k$.}
\label{fig:mp_application1}
\end{figure}
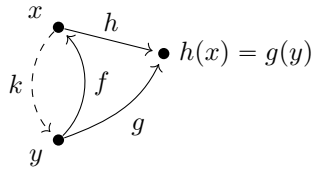

First, we will show that $h$ is Turing-invariant. Suppose $x_1 \equiv_T x_2$ are in the domain of $h$. By definition of $h$, there are reals $y_1$ and $y_2$ such that $f(y_1) \equiv_T x_1$, $f(y_2) \equiv_T x_2$, $h(x_1) = g(y_1)$ and $h(x_2) = g(y_2)$ (formally, $y_1 = k(\tilde{x}_1)$ and $y_2 = k(\tilde{x}_2)$). Therefore $f(y_1) \equiv_T f(y_2)$ and so
\[
  h(x_1) = g(y_1) \equiv_T g(y_2) = h(x_2)
\]
by our assumption about $f$ and $g$.

Note that this argument also implies that for any $y$ for which $f(y)$ is in the domain of $h$, $h(f(y)) \equiv_T g(y)$.

Next, we will show that $h$ is measure-preserving. Let $a$ be an arbitrary degree. We want to show that $h$ gets above $a$ on a cone. By determinacy, it is enough to show that it gets above $a$ cofinally. So let $b$ be an arbitrary real, and we will show that $h$ is above $a$ on some degree above $b$. We claim that $f(a\oplus b)$ is one such degree. Since $f$ is above the identity, $f(a\oplus b)$ is above $b$. By the observation we have already made about $h$, $h(f(a\oplus b)) \equiv_T g(a\oplus b)$. Since $g$ is above the identity, this implies $h(f(a\oplus b))$ is above $a$.

Since $h$ is measure-preserving, Theorem~\ref{thm:mp_part1measurepreserving} implies that $h$ is above the identity on a cone\footnote{One might object that Theorem~\ref{thm:mp_part1measurepreserving} was stated for functions defined on all of $\Cantor$ while $h$ is only defined on a cone. However, the proof of that theorem works just as well for functions defined on a cone, or even a pointed perfect tree.}. We will now use this fact to show that $f$ is below $g$ on a cone. Let $x$ be any degree in a cone on which $h$ is above the identity. Since $f$ is above the identity, $f(x)$ is above $x$. Since $x$ was in a cone on which $h$ is above the identity, this implies that $h(f(x)) \ge_T f(x)$. Since $h(f(x)) \equiv_T g(x)$, we have shown that $g(x) \ge_T f(x)$, as desired.
\end{proof}

One might hope to use the theorem above to show that the Martin order is linear above the identity by showing that for every pair of Turing-invariant functions $f$ and $g$ which are above the identity, the relationship required by the theorem holds for $f$ and $g$ in some order (or at least, holds on a cone).

At first, this does not seem like such an unreasonable hope. It does hold for many pairs of functions on the Turing degrees. For example, if two Turing degrees have the same Turing jump then they also have the same $\omega$-jump. Likewise, if they have the same $\omega$-jump then they also have the same hyperjump. But if we go just a little bit higher than the hyperjump, we can find examples of pairs of Turing-invariant functions which do not have this sort of relationship. We give one such example below.

\begin{example}
Let $f \colon 2^\omega \to 2^\omega$ be the hyperjump, i.e.\ $f(x) = \hyperjump^x$, and let $g \colon 2^\omega \to 2^\omega$ be the function defined by
\[
g(x) = (\hyperjump^x)^{(\omega_1^x)}.
\]
The function $g$ is well-defined because $\omega_1^{\hyperjump^x}$ is always strictly greater than $\omega_1^x$ and thus the $\omega_1^x$-th jump of $\hyperjump^x$ is well-defined.

On the one hand, it is easy to see that there are reals $x$ and $y$ such that $g(x) \equiv_T g(y)$ but $f(x) \nequiv_T f(y)$. To find such $x$ and $y$ we can take reals $\tilde{x}$ and $\tilde{y}$ which are above Kleene's $\hyperjump$ and not Turing equivalent but such that $\tilde{x}^{(\omega_1^{\text{CK}})} \equiv_T \tilde{y}^{(\omega_1^\text{CK})}$ and then use hyperjump inversion to find $x$ and $y$ which are low for $\omega_1^\text{CK}$ such that $\hyperjump^x = \tilde{x}$ and $\hyperjump^y = \tilde{y}$.

On the other hand, we can \emph{also} find reals $x$ and $y$ such that $f(x) \equiv_T f(y)$ but $g(x) \nequiv_T g(y)$. To see why, let $x$ be some real such that $\omega_1^x > \omega_1^{\text{CK}}$. We can use hyperjump inversion to find a real $y$ such that $\omega_1^y = \omega_1^{\text{CK}}$ and $\hyperjump^y \equiv_T \hyperjump^x$. Since $\omega_1^x \neq \omega_1^y$, it is clear that $(\hyperjump^x)^{(\omega_1^x)} \nequiv_T (\hyperjump^y)^{(\omega_1^y)}$.

Actually, without even bothering to construct these examples, it should have been apparent that $f$ and $g$ cannot have the relationship required by the theorem above. Since $f <_M g$, we know that there must be $x$ and $y$ such that $g(x) \equiv_T g(y)$ and $f(x) \nequiv_T f(y)$ since otherwise the theorem would imply that $g$ is below $f$ on a cone. On the other hand, if we look at the proof of the theorem we can also see it provides reason to believe that there are reals $x$ and $y$ such that $f(x) \equiv_T f(y)$ and $g(x) \nequiv_T g(y)$. If not, then the proof of the theorem would imply that there is a Turing-invariant function $h$ such that $h(f(x)) = g(x)$ on a cone. Such an $h$ would have to be above every function of the form $x \mapsto x^{(\alpha)}$ for a fixed countable ordinal $\alpha$, but also below the hyperjump. But it seems plausible that no such function exists at all and it is known that such a function cannot be uniformly invariant or order-preserving so we should not expect to be able to find it so easily.
\end{example}

\section{Part 1 of Martin's Conjecture for Order-Preserving Functions}
\label{sec:op}

In this section, we will prove part 1 of Martin's Conjecture for order-preserving functions. We will do so by first proving that every order-preserving function is either constant on a cone or measure-preserving and then obtain part 1 of Martin's Conjecture for order-preserving functions as a corollary to part 1 of Martin's Conjecture for measure-preserving functions.

Since our proof of part 1 of Martin's Conjecture for Borel measure-preserving functions required more determinacy than is provable in $\ZF$ (in particular, $\mathbf{\Pi}^1_1$ determinacy), this does not give us a proof of part 1 of Martin's Conjecture for Borel order-preserving functions in $\ZF$. We give such a proof using an idea due to Takayuki Kihara combined with our result that nontrivial order-preserving functions are measure-preserving (which is provable in $\ZF$ for Borel functions).

We will finish this section with an application of our results to the theory of locally countable Borel quasi-orders.

\subsection{A theorem on perfect sets}
\label{sec:op_perfect}
A key step in our proof that order-preserving functions are either constant on a cone or measure-preserving is a technical theorem about perfect sets, which we will prove below. The theorem was inspired by, and is a strengthening of, a theorem proved by Groszek and Slaman in \cite{groszek1998basis}, which we will state next.

\begin{definition}
Suppose that $A$ is a perfect subset of $2^\omega$ and $x \in A$. Say that $x$ is \term{eventually constant in $A$} if there is some $n$ such that
\[
\forall y \in A\, (x\restriction n = y\restriction n \implies x \leq y) \quad\text{or}\quad \forall y \in A\, (x\restriction n = y\restriction n \implies y \leq x)
\]
where the ordering is the usual lexicographic ordering on $2^\omega$.
\end{definition}

If you think of $A$ as the set of branches through a perfect tree, this is saying that $x$ eventually either always goes to the left or always goes to the right in the tree.

\begin{theorem}[Groszek and Slaman \cite{groszek1998basis} lemma 2.2]
Suppose that $A$ is a perfect subset of $2^\omega$, $B$ is a countable dense subset of $A$ which contains no element which is eventually constant in $A$, and $\seq{c_i}_{i \in \N}$ is a countable sequence which contains every element of $B$. Then for every $x$ there are $y_0$ and $y_1$ in $A$ such that
\[
\left(\bigoplus_{i \in \N}c_i\right) \oplus y_0 \oplus y_1 \ge_T x.
\]
\end{theorem}

The main shortcoming of Groszek and Slaman's theorem for our purposes is that to compute $x$ you need to be able to compute the countable sequence $\seq{c_i}_{i \in \omega}$. In the situation where we would like to use the theorem we can only be assured of having some real which computes every element of the sequence, but does not necessarily compute the sequence itself (i.e.\ it may compute the sequence in a non-uniform way). The theorem we prove below was formulated to fix this problem.

To prove our strengthened version of Groszek and Slaman's theorem, we use a proof that is somewhat different from theirs (and which is essentially a souped-up version of the coding argument used in the first author's proof of Martin's Conjecture for regressive functions on the hyperarithmetic degrees~\cite{lutz2023martins}). This proof also allows us to get rid of the requirement that no element of $\seq{c_i}$ is eventually constant in $A$ (though this is mostly just a cosmetic improvement). Roughly speaking, here's what is different about our proof. In Groszek and Slaman's proof, they start with a real $x$ which they want to code using two elements of $A$. To do so, they code the bits of $A$ into the sequence of decisions about whether to turn left or right in the tree whose branches are $A$. In our proof, we instead essentially code the bits of $x$ into the Kolmogorov complexity of initial segments of the elements of $A$.

\begin{theorem}
\label{thm:basis}
Suppose that $A$ is a perfect subset of $2^\omega$, $B$ is a countable dense subset of $A$, and $c$ is a real which computes each $b \in B$. Then for every $x$ there are $y_0,y_1,y_2,y_3$ in $A$ such that
\[
c\oplus y_0\oplus y_1\oplus y_2\oplus y_3 \ge_T x.
\]
\end{theorem}

\begin{remark*}
The proof of this theorem is the kind of thing that is not that hard to explain on a blackboard during a one-on-one conversation, but which looks quite complicated when all the details are written down. In the proof below, we have tried as best we can to explain the idea of the construction without getting lost in the messy details. If we have succeeded, then the construction should not actually seem so complicated. If we have failed then we hope the reader will forgive us.
\end{remark*}

\begin{proof}
The basic idea here is to build up $y_0,\ldots,y_3$ by finite extensions and on each step code one more bit of $x$. To use $c$ together with $y_0,\ldots,y_3$ to compute $x$, we have to decode the results of this coding process: figure out what happened on each step and recover the bits of $x$ as a consequence.

Perhaps you could imagine that we have two rooms which are completely separated from each other. In the first room, someone---let's call them the \term{coder}---is given the real $x$ (and whatever other information they need) and tasked with building the $y_i$'s one bit at a time. In the other room, the coder's friend---let's call them the \term{decoder}---is given $c$ and then receives the bits of the $y_i$'s one at a time, and needs to reconstruct what the coder did. So the coder needs to not only encode the bits of $x$, but also encode enough extra information to allow the decoder to reconstruct the coding process. Note, by the way, that the decoder's process needs to be computable, but there is no such requirement on the coder.

There is also one more constraint: the coder needs to end up building elements of the set $A$. They can accomplish this by making sure that on each step, the portions of the $y_i$'s that they have built so far are each consistent with some element of $B$. This will ensure that each $y_i$ is the limit of a sequence of elements of $A$, and thus each $y_i$ is in $A$ since $A$ is closed.

The most natural way to describe all this is to describe how these two processes work together. That is, describe what the coder is doing on a single step of the process, and, at the same time, describe what the decoder is doing on the same step. In particular, we will assume that the decoder has so far reconstructed all the steps correctly and see how they can also reconstruct the next step correctly. Actually, the decoder will not completely reconstruct everything the coder does, but just enough to decode the next bit of $x$ and to allow the decoding process to continue on the next step. We will now begin to describe what happens on a single step in both processes.

\medskip\noindent\textbf{The situation after $n$ steps.} 
Suppose the coder has just finished the $n^\text{th}$ step of the coding process. In other words, they have formed finite initial segments $y_0^n, y_1^n, y_2^n, y_3^n$ of $y_0,y_1,y_2,y_3$, respectively. And to make sure that the reals being built will end up in $A$, they have also picked elements $b_0^n,b_1^n,b_2^n,b_3^n$ of $B$ such that each $y_i^n$ is an initial segment of $b_i^n$. They now want to code the $(n + 1)^\text{th}$ bit of $x$ by extending each of $y_0^n,y_1^n,y_2^n,y_3^n$ by a finite amount, making sure each one is still an initial segment of some element of $B$ (though perhaps not the same element as at the end of the $n^\text{th}$ step), while also giving the decoder enough information to recover what happened on this step.

Let's now consider what things look like for the decoder. At the end of the $n^\text{th}$ step of the decoding process, the decoder has a guess about exactly two of the $b_0^n, b_1^n, b_2^n, b_3^n$. In particular, if $n$ is even then the decoder has a guess about $b_0^n$ and $b_1^n$ and if $n$ is odd then the decoder has a guess about $b_2^n$ and $b_3^n$.

From now on, we will assume that $n$ is even and thus that the decoder has a guess about $b_0^n$ and $b_1^n$. This guess takes the form of two numbers, $e_0^n$ and $e_1^n$. These should be thought of as indices for programs which use $c$ to compute $b_0^n$ and $b_1^n$, respectively. In other words, the decoder is guessing that $\Phi_{e_0^n}(c)$ and $\Phi_{e_1^n}(c)$ are both total and are equal to $b_0^n$ and $b_1^n$, respectively. Figure~\ref{fig:view} below shows what the coder has built and what the decoder knows and has guessed at the end of step $n$.

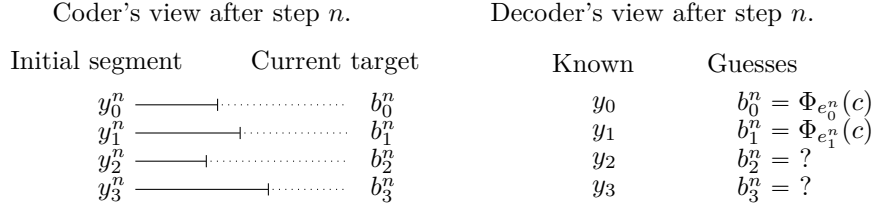
\begin{figure}[h!]
\centering
\begin{tikzpicture}[scale=0.75]
\node[left] (coder) at (4, 1.6) {Coder's view after step $n$.};
\node[left] (decoder) at (12.2, 1.6) {Decoder's view after step $n$.};

\node[left] (is) at (1, .75) {Initial segment};
\node[left] (t) at (5.2, .75) {Current target};

\node[left] (s0) at (0, 0) {$y_0^n$};
\node[left] (s1) at (0, -.5) {$y_1^n$};
\node[left] (s2) at (0, -1) {$y_2^n$};
\node[left] (s3) at (0, -1.5) {$y_3^n$};  
\node[inner sep=0] (t0) at (1.5, 0) {};
\node[inner sep=0] (t1) at (1.9, -.5) {};
\node[inner sep=0] (t2) at (1.3, -1) {};
\node[inner sep=0] (t3) at (2.4, -1.5) {};

\node[right] (b0) at (4, 0) {$b_0^n$};
\node[right] (b1) at (4, -.5) {$b_1^n$};
\node[right] (b2) at (4, -1) {$b_2^n$};
\node[right] (b3) at (4, -1.5) {$b_3^n$};

\draw[-|] (s0) -- (t0);
\draw[-|] (s1) -- (t1);
\draw[-|] (s2) -- (t2);
\draw[-|] (s3) -- (t3);
\draw[dotted, shorten >=5pt] (t0) -- (b0);
\draw[dotted, shorten >=5pt] (t1) -- (b1);
\draw[dotted, shorten >=5pt] (t2) -- (b2);
\draw[dotted, shorten >=5pt] (t3) -- (b3);

\node[left] (y) at (9, 0.75) {Known};
\node[left] (e) at (8.5 + 3.4, 0.75) {Guesses};

\node[left] (y0) at (8.7, 0) {$y_0$};
\node[left] (y1) at (8.7, -.5) {$y_1$};
\node[left] (y2) at (8.7, -1) {$y_2$};
\node[left] (y3) at (8.7, -1.5) {$y_3$};  

\node[right] (db0) at (7.5 + 3, 0) {$b_0^n = \Phi_{e_0^n}(c)$};
\node[right] (db1) at (7.5+3, -.5) {$b_1^n = \Phi_{e_1^n}(c)$};
\node[right] (db2) at (7.5+3, -1) {$b_2^n = $ ?};
\node[right] (db3) at (7.5+3, -1.5) {$b_3^n =$ ?};
\end{tikzpicture}
\caption{The coder and decoder's view at the end of step $n$.}
\label{fig:view}
\end{figure}

We will now assume that the decoder's guesses at the end of step $n$ are correct, describe what happens on the next step of the coding and decoding processes and argue that the decoder's guess at the end of this next step is still correct.

\medskip\noindent\textbf{The decoding process.}
Let's describe what happens on the decoding side first. The decoder first does the following:
\begin{enumerate}
\item Look at more and more bits of $y_0$ until they find a place where $y_0$ disagrees with $\Phi_{e_0^n}(c)$---in other words, a place where $y_0$ disagrees with $b_0^n$. Let $l_0$ be the first such position.
\item Repeat this process with $y_1$ and $\Phi_{e_1^n}(c)$ to find the first position $l_1$ where they disagree.
\end{enumerate}

Next, the decoder uses $l_0$ and $l_1$ to determine guesses $e_2^{n + 1}$ and $e_3^{n + 1}$ for $b_2^{n + 1}$ and $b_3^{n + 1}$. To pick $e_2^{n + 1}$, the decoder takes the least $e$ such that when $\Phi_{e}(c)$ is run for at most $l_1$ steps, it converges and agrees with $y_2$ on all inputs less than $l_0$. In other words, the least $e$ such that for all $m \leq l_0$,
\[
\Phi_e(c, m)[l_1]\downarrow = y_2(m).
\]
The decoder then picks $e_3^{n + 1}$ in the same way.

The decoder then extracts the $(n + 1)^\text{th}$ bit of $x$ by comparing $e_2^{n + 1}$ and $e_3^{n + 1}$. If $e_2^{n + 1}$ is smaller than $e_3^{n + 1}$, the decoder guesses that the $(n + 1)^\text{th}$ bit of $x$ is $0$. Otherwise they guess that it is $1$. The entire decoding process is pictured in Figure~\ref{fig:decoding} below.

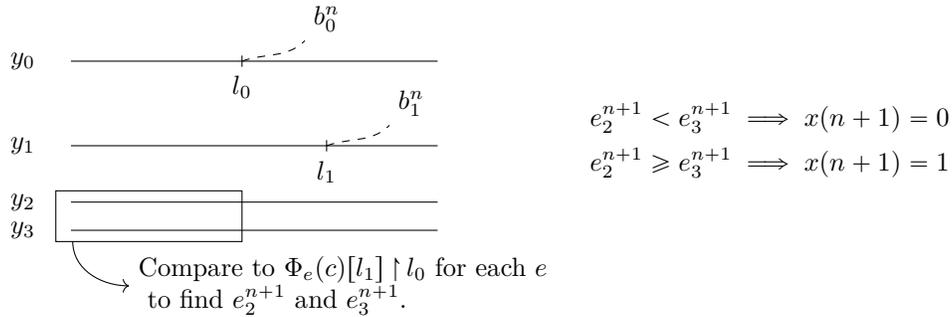
\begin{figure}[h!]
\centering
\begin{tikzpicture}[scale=0.75]

\node[left] (y0) at (-3, 0) {$y_0$};
\node (b0) at (2, 0.75) {$b_0^n$};  
\node[left] (y1) at (-3, -1.5) {$y_1$};
\node (b1) at (3.5, -0.75) {$b_1^n$};
\node[left] (y2) at (-3, -2.5) {$y_2$};
\node[left] (y3) at (-3, -3) {$y_3$};  

\node[inner sep=0] (t0) at (4, 0) {};
\node[inner sep=0] (t1) at (4, -1.5) {};
\node[inner sep=0] (t2) at (4, -2.5) {};
\node[inner sep=0] (t3) at (4, -3) {};

\node[below] (l0) at (0.5, -0.1) {$l_0$};
\draw[-] (0.5, -0.1) -- (0.5, .1);
\draw[dashed] (0.5, 0) to[out=20, in=-135] (b0);

\node[below] (l1) at (2, -1.6) {$l_1$};
\draw[-] (2, -1.6) -- (2, -1.4);
\draw[dashed] (2, -1.5) to[out=20, in=-135] (b1);

\draw[-, shorten <=10pt] (y0) -- (t0);
\draw[-, shorten <=10pt] (y1) -- (t1);
\draw[-, shorten <=10pt] (y2) -- (t2);
\draw[-, shorten <=10pt] (y3) -- (t3);

\draw[draw=black] (-2.8, -3.2) rectangle (0.5, -2.3);
\node[below right, align=left] (e) at (-1.5, -3.3) {Compare to $\Phi_e(c)[l_1] \uh l_0$ for each $e$\\\ to find $e_2^{n + 1}$ and $e_3^{n + 1}$.};
\draw[->] (-2.5, -3.2) to[out=-90, in=180] (e);

\node[right] (a) at (6.5,-1) {$e_2^{n + 1} < e_3^{n + 1} \implies x(n + 1) = 0$};
\node[right] (b) at (6.5,-1.8) {$e_2^{n + 1} \geq e_3^{n + 1} \implies x(n + 1) = 1$};
\end{tikzpicture}
\caption{Step $n + 1$ of the decoding process.}
\label{fig:decoding}
\end{figure}

This completes our description of the decoding process. If we have done a good job, then the reader should be able to fill in all the details about the coding process for themselves. But we will describe them here anyways for the sake of completeness.

\medskip\noindent\textbf{The coding process.} At the end of step $n$, the coder knows exactly what the decoder's guesses are (after all, they have access to all the same information as the decoder, so they can figure out exactly what the decoder did on each previous step). The coder also knows the next bit of $x$ that they need to code.

Suppose for convenience that the next bit is a $0$. The coder will begin by choosing some value which they will make sure is the decoder's guess $e_2^{n + 1}$ for $b_2^{n + 1}$. One choice that works well enough is to simply let $e_2^{n + 1}$ be the first $e$ such that $\Phi_e(c)$ is total and equal to $b_2^n$ (which is guaranteed to exist because $b_2^n \leq_T c$ by assumption). So the coder can simply set $b_2^{n + 1}$ to be $b_2^n$ (i.e.\ they will not change $b_2^n$ on this step).

Next, the coder wants to choose some value which they will ensure is the decoder's guess $e_3^{n + 1}$ for $b_3^{n + 1}$. In particular, since bit $n + 1$ of $x$ is a $0$, the coder wants to make sure that $e_3^{n + 1}$ is larger than $e_2^{n + 1}$. This is also easy enough to accomplish. Since there are infinitely many elements of $B$ which extend $y_3^n$ but only finitely many programs less than $e_2^{n + 1}$, there must be some $b \in B$ such that the least $e$ for which $\Phi_e(c) = b$ is greater than $e_2^{n + 1}$. The coder then sets $b_3^{n + 1}$ to be such a $b$ and sets $e_3^{n + 1}$ to be the least $e$ such that $\Phi_e(c) = b_3^{n + 1}$.

Now the coder needs to make sure that the values $l_0$ and $l_1$ that the decoder recovers are large enough that the decoder's guesses $e_2^{n + 1}$ and $e_3^{n + 1}$ are correct. To do so, the coder picks $l_0$ to be large enough that $e_2^{n + 1}$ and $e_3^{n + 1}$ are the first indices $e$ and $e'$ for which $\Phi_e(c)$ and $\Phi_{e'}(c)$ converge on all inputs less than $l_0$ and agree with $b_2^{n + 1}$ and $b_3^{n + 1}$, respectively, on all such inputs. The coder can then choose $l_1$ to be large enough for $\Phi_{e_2^{n + 1}}(c)$ and $\Phi_{e_3^{n + 1}}(c)$ to both converge on all inputs less than $l_0$.

Next, the coder can choose $b_0^{n + 1}$ to be an element of $B$ which has $y_0^n$ as an initial segment and agrees with $b_0^n$ on the first $l_0$ bits, but which eventually disagrees with $b_0^n$. They should retroactively increase the value of $l_0$ to the first position of disagreement between $b_0^n$ and $b_0^{n + 1}$, which may necessitate increasing $l_1$ as well (so that the programs $e_2^{n + 1}$ and $e_3^{n + 1}$ are given enough time to converge on the first $l_0$ inputs).

The coder can now choose $b_1^{n + 1}$ in a similar manner, to be some element of $B$ which has $y_1^n$ as an initial segment and agrees with $b_1^n$ on the first $n$ bits, but eventually disagrees. They should then retroactively increase $l_1$ to this first position of disagreement. Notice that increasing $l_0$ and $l_1$ in this way is harmless.

Now the coder defines
\begin{align*}
y_0^{n + 1} &:= b_0^{n + 1}\restriction l_0\\
y_1^{n + 1} &:= b_1^{n + 1}\restriction l_1\\
y_2^{n + 1} &:= b_2^{n + 1}\restriction l_0\\
y_3^{n + 1} &:= b_3^{n + 1}\restriction l_0.
\end{align*}
to make sure that $y_0^{n + 1},\ldots,y_3^{n + 1}$ match $b_0^{n + 1},\ldots,b_3^{n + 1}$ for enough inputs for the decoder's process to work as the coder intends.

That completes our description of the coding process. The main thing to notice is that we have managed to describe actions for the coder which guarantee that if the decoder follows the decoding process described above, then all of their guesses at the end of step $n + 1$ are correct and they have also correctly extracted the $(n + 1)^\text{th}$ bit of $x$.
\end{proof}

\subsubsection*{A Useful Corollary}

We now prove a corollary of theorem \ref{thm:basis} that will allow us to give a more streamlined proof that every order-preserving function is either constant on a cone or measure-preserving.

\begin{definition}
A subset $A$ of $2^\omega$ is \term{countably directed for Turing reducibility} if for all countable subsets $B \subseteq A$, $A$ contains an upper bound on the Turing degrees of the elements of $B$---i.e.\ there is some $x \in A$ such that for all $y \in B$, $y \le_T x$.
\end{definition}

\begin{corollary}
\label{cor:op_directed}
If $A$ is a subset of $2^\omega$ which is countably directed for Turing reducibility and contains a perfect set then $A$ is cofinal.
\end{corollary}

\begin{proof}
We start with a set $A$ which contains a perfect set and which is countably directed. Let $P$ be a perfect set contained in $A$ and let $B$ be a countable dense subset of $P$. Since $A$ is countably directed, we can find some $c$ in $A$ which is an upper bound for $B$.

Our goal is to show that $A$ is cofinal. So we start with an arbitrary $x$ and we want to find some $y$ in $A$ that computes $x$. By Theorem~\ref{thm:basis}, we can find reals $y_0, y_1, y_2, y_3$ in $P$ (and therefore also in $A$) such that
\[
c \oplus y_0\oplus y_1\oplus y_2\oplus y_3 \ge_T x.
\]
Since $A$ is countably directed, we can find an upper bound for $c, y_0, y_1, y_2, y_3$ in $A$. This upper bound obviously computes $x$, so it is the $y$ we are after.
\end{proof}

\subsection{Order-preserving functions are measure-preserving}
\label{sec:op_mp}
In this section, we will prove that part 1 of Martin's Conjecture holds for all order-preserving functions. We will do this by proving that every order-preserving function is either constant on a cone or measure-preserving and then invoking Theorem~\ref{thm:mp_part1measurepreserving}, which states that part 1 of Martin's Conjecture holds for all measure-preserving functions.

To prove that every order-preserving function is either measure-preserving or constant on a cone we will use the theorem on perfect sets that we proved in the previous section, together with the fact that, under $\AD$, every set of reals is either countable or contains a perfect set (known as the \term{perfect set theorem}, see \cite{jech2003set}, Theorem 33.3).

\begin{theorem}[$\ZF + \AD$]
\label{thm:op_mp}
If $f \colon 2^\omega \to 2^\omega$ is an order-preserving function then $f$ is either constant on a cone or measure-preserving.
\end{theorem}

\begin{proof}
Before giving the proof in detail, here's a sketch. Suppose that $f\colon 2^\omega \to 2^\omega$ is an order-preserving function. By the perfect set theorem, $\range(f)$ is either countable or contains a perfect set. We will show that if $\range(f)$ is countable then $f$ is constant on a cone and if $\range(f)$ contains a perfect set then $f$ is measure-preserving.

The case where $\range(f)$ is countable is straightforward. In the case where $\range(f)$ contains a perfect set we will use the theorem on perfect sets that we proved in the previous section (more specifically, we will use Corollary~\ref{cor:op_directed}). The key point is that since $f$ is order-preserving, its range is countably directed for Turing reducibility.

\medskip\noindent
\textbf{Case 1: the range of $f$ is countable.} In this case, we can write $\Cantor$ as a countable union of sets on which $f$ is constant (i.e.\ the preimages of points in the range of $f$). Since there are only countably many of these sets, Corollary~\ref{cor:framework_pointed} implies that at least one of them contains $[T]$ for some pointed perfect tree $T$. Thus $f$ is constant on $[T]$ and hence constant on a cone.

\medskip\noindent
\textbf{Case 2: the range of $f$ contains a perfect set.} The main point here is that the range of an order-preserving function is countably directed. To see why, suppose that $x_0,x_1,\ldots$ are all in the range of $f$. Pick reals $y_0, y_1, \ldots$ such that $f(y_0) = x_0$, $f(y_1) = x_1$, and so on. Let $y$ be the Turing join of all the $y_i$'s. Since $y$ computes each $y_i$ and $f$ is order-preserving, $f(y)$ computes each $x_i$. In other words, $f(y)$ is an upper bound for $\{x_0,x_1,\ldots\}$.

Since $\range(f)$ contains a perfect set and is countably directed, Corollary~\ref{cor:op_directed} implies that it is cofinal. Now we want to show $f$ is measure-preserving. In other words, we start with an arbitrary $a$ and we want to show that there is some $b$ so that $f$ sends everything in the cone above $b$ into the cone above $a$. Since the range of $f$ is cofinal, there is some $x \ge_T a$ in the range of $f$. Since $x$ is in the range of $f$, there is some $y$ such that $f(y) = x$. And since $f$ is order-preserving, it takes the cone above $y$ into the cone above $x$ and hence into the cone above $a$.
\end{proof}

\begin{remark}
\label{rmk:op_mp_borel}
It will be useful later to note that if $f$ is Borel then the above proof can be carried out in $\ZF$ (with no extra determinacy hypotheses). To see why, observe that the proof only uses determinacy in two places: first, to conclude that $\range(f)$ is either countable or contains a perfect set and second, in case 1 (the case where $\range(f)$ is countable) in the form of Corollary~\ref{cor:framework_pointed}. If $f$ is Borel then $\range(f)$ is analytic and hence it follows from the perfect set theorem for analytic sets (see \cite{kechris1995classical}, Exercise 14.13) that $\range(f)$ is either countable or contains a perfect set. Furthermore, Corollary~\ref{cor:framework_pointed} is applied to sets of the form $f^{-1}(x)$ and if $f$ is Borel then so are these sets and hence Corollary~\ref{cor:framework_pointed} only requires Borel determinacy.
\end{remark}

\begin{theorem}[$\ZF + \AD + \DC_\R$]
Part 1 of Martin's Conjecture holds for all order-preserving functions.
\end{theorem}

\begin{proof}
Suppose $f \colon 2^\omega \to 2^\omega$ is a Turing-invariant function which is order-preserving. We want to show that $f$ is either constant on a cone or above the identity on a cone. By Theorem \ref{thm:op_mp}, $f$ is either constant on a cone or measure-preserving. If $f$ is constant on a cone then we are done and if $f$ is measure-preserving then by Theorem \ref{thm:mp_part1measurepreserving} it is above the identity on a cone.
\end{proof}

\subsection{A proof for Borel functions that works in $\ZF$}
\label{sec:op_borel}
In the previous section, we proved part 1 of Martin's Conjecture for order-preserving functions by reducing to the case of measure-preserving functions. However, as discussed in section~\ref{sec:mp_proof2}, we do not currently know how to prove part 1 of Martin's Conjecture for Borel measure-preserving functions in $\ZF$, only in $\ZF + \mathbf{\Pi^1_1}-\Det$. Thus our proof in the previous section only implies that part 1 of Martin's Conjecture for order-preserving Borel functions holds in $\ZF + \mathbf{\Pi^1_1}-\Det$. In this section, we will prove that it also holds in $\ZF$.

The key idea in this section is due to Kihara~\cite{kihara2021personal}, building on an observation by Carroy. Kihara observed that a theorem in descriptive set theory known as the \term{Solecki dichotomy} can be used to prove a weak form of part 1 of Martin's Conjecture for order-preserving functions (earlier, Carroy had suggested that the Solecki dichotomy and Martin's Conjecture are related). In particular, Kihara showed that if $f$ is an order-preserving function then either
\begin{enumerate}
\item $f(x) \leq_T x$ on a cone (in other words, $f$ is regressive)
\item or there is some real $a$ such that $x' \leq_T f(x) \oplus a$ on a cone.
\end{enumerate}

Furthermore, when combined with our results and a result of Slaman and Steel, the statement obtained by Kihara can be used to prove the full part 1 of Martin's Conjecture for order-preserving functions. More precisely, in the first case above, Slaman and Steel's theorem on Martin's Conjecture for regressive functions shows that either $f$ is constant on a cone or equal to the identity on a cone. In the second case, our result that order-preserving functions are either constant on a cone or measure-preserving shows that $f(x)$ computes $a$ on a cone (note that $f$ cannot be constant on a cone in this case) and thus that $f(x)$ computes $x'$ on a cone. 

It is possible to check that if $f$ is Borel then all parts of this argument---Kihara's use of the Solecki dichotomy, Slaman and Steel's result and our result---can all be proved in $\ZF$ alone and thus this yields a $\ZF$ proof of part 1 of Martin's Conjecture for order-preserving Borel functions. Moreover, the whole proof also works in $\ZF + \AD$ when $f$ is arbitrary (without needing $\DC_\R$).

Instead of presenting Kihara's original argument, we will use a simplified version, which is also more self-contained. The simplified version of Kihara's argument does not use the Solecki dichotomy, but a more elementary statement which we will refer to as the ``baby Solecki dichotomy'' (first proved by Carroy in~\cite{carroy2013quasi}). In the rest of this section, we will explain this statement, give a self-contained proof of it and then explain how it can be used to give a $\ZF$ proof of part 1 of Martin's Conjecture for order-preserving Borel functions.

\subsubsection*{The Solecki dichotomy and the baby Solecki dichotomy}

Informally, the Solecki dichotomy states that for every sufficiently definable function $f\colon \Cantor \to \Cantor$, either $f$ is a countable union of continuous functions or the Turing jump (as a function on $\Cantor$) is reducible to $f$. To state this formally, we need to give precise definitions of ``a countable union of continuous functions'' and ``the Turing jump is reducible to $f$.''

\begin{definition}
A function $f\colon \Cantor \to \Cantor$ is \term{$\sigma$-continuous} if there is a countable partition $\seq{A_n}_{n \in \N}$ of $\Cantor$ such that for each $n$, $f\restriction_{A_n}$ is continuous with respect to the subspace topology on $A_n$.
\end{definition}

Note that there is a small subtlety here: just because $f\uh_{A_n}$ is continuous with respect to the subspace topology on $A_n$ does not mean that $f\uh_{A_n}$ can be extended to a continuous function defined on all of $\Cantor$. We will also refer to a partial function which is continuous with respect to the subspace topology on its domain as a \term{partial continuous function}.

\begin{definition}
Given functions $f, g \colon \Cantor \to \Cantor$, $f$ is \term{continuously reducible}\footnote{This notion of reducibility has also been called \term{strong continuous Weihrauch reducibility} \cite{brattka2011weihrauch}.} to $g$, written $f \leq_c g$, if there are partial continuous functions $\phi, \psi \colon \Cantor \to \Cantor$ such that for all $x \in \Cantor$, $f(x) = \psi(g(\phi(x)))$. In other words, the following diagram commutes
\begin{center}
\begin{tikzcd}
\Cantor \arrow[r, "g"]
& \Cantor \arrow[d, "\psi"]\\
\Cantor \arrow[r, "f"] \arrow[u, "\phi"]
& \Cantor
\end{tikzcd}
\end{center}
\end{definition}

\begin{theorem}[Solecki dichotomy]
\label{thm:solecki}
For every Borel function $f \colon \omega^\omega \to \omega^\omega$, either $f$ is $\sigma$-continuous or $j \leq_c f$ (where $j$ denotes the Turing jump).
\end{theorem}

Theorem~\ref{thm:solecki} was first proved by Solecki in \cite{solecki1998decomposing} in the special case where $f$ is of Baire class $1$. It was extended to all Borel functions by Zapletal in \cite{zapletal2004descriptive} and to all analytic functions by Pawlikowski and Sabok in \cite{pawlikowski2012decomposing}. It is also known to hold for all functions under $\AD$ \cite{zapletal2004descriptive}.

Actually, the Solecki dichotomy as it is usually stated differs from Theorem~\ref{thm:solecki} above in two ways. First, it is usually stated with the Turing jump replaced by a different function known as the Pawlikowski function. Since this function is continuously equivalent to the Turing jump, this does not change the content of the theorem. Second, it is usually stated with continuous reducibility replaced by a stronger reducibility notion known as topological embeddability. However, the distinction between topological embeddability and continuous reducibility does not matter in Kihara's proof so we will ignore this difference.

As we mentioned above, Kihara's proof relies on the Solecki dichotomy, but can be modified to rely on a more elementary statement in descriptive set theory, which we will refer to as the ``baby Solecki dichotomy.'' Roughly speaking, this statement says that for every sufficiently definable function $f \colon \Cantor \to \Cantor$, either $f$ is a countable union of constant functions or the identity function is reducible to $f$. Just as the Solecki dichotomy characterizes those functions which the Turing jump is reducible to, this statement characterizes those functions which the identity is reducible to. 

\begin{theorem}[Baby Solecki dichotomy]
\label{thm:baby_solecki}
For every Borel function $f \colon \Cantor \to \Cantor$, either $\range(f)$ is countable or $\id \leq_c f$.
\end{theorem}

This theorem was first proved by Carroy in \cite[Proposition 2.1]{carroy2013quasi}. His proof works by invoking Silver’s theorem that every analytic equivalence relation on $\Cantor$ either has countably many equivalence classes or a perfect set of pairwise nonequivalent elements. However, it is also possible to give a proof using a modified version of the perfect set game (essentially the ``unfolded'' perfect set game), which is the proof we give below. To the best of our knowledge, this proof is new. Also, it can be easily generalized to yield the same conclusion for all functions under $\ZF + \AD$, whereas the generalization of Carroy's proof seems to require $\AD^+$.

\begin{proof}
Let us assume that $\range(f)$ is not countable and prove $\id \leq_c f$. First observe that it is enough to show that there is a perfect set $P$ such that $f\uh_P$ is both continuous and injective. Here's why. Any perfect set $P \subseteq \Cantor$ is homeomorphic to $\Cantor$. Fix such a homeomorphism $\phi \colon \Cantor \to P$. Furthermore, since $f\uh_P$ is a continuous, injective function on a compact set, it has a continuous left inverse, $\psi \colon \range(f) \to P$. Thus for any $x \in \Cantor$, we have
\begin{align*}
  \phi^{-1}\circ \psi \circ f\circ \phi(x) &= \phi^{-1}(\psi \circ f (\phi(x)))\\
  &= \phi^{-1}(\phi(x)) &\text{$\psi$ is a left inverse of $f\uh_P$}\\
  &= x.
\end{align*}
So the partial continuous functions $\phi$ and $\phi^{-1}\circ \psi$ witness that $\id \leq_c f$.

Now let us show that such a set $P$ exists. To do so, we will define a game, similar to the perfect set game and show that if player 1 has a winning strategy in this game then there is a perfect set on which $f$ is continuous and injective and if player 2 has a winning strategy then $\range(f)$ is countable.

\medskip\noindent\textbf{Informal description of the game.}
We will first give an informal description of the game. First, player 1 plays two pairs of strings $\ip{\sigma_0}{\tau_0}$ and $\ip{\sigma_1}{\tau_1}$ such that $\sigma_0, \sigma_1$ are incompatible and $\tau_0, \tau_1$ are incompatible. These should be thought of as two different options for initial segments of reals $x$ and $y$ such that $f(x) = y$. In other words, $\sigma_0$ and $\tau_0$ are one option for initial segments of $x$ and $y$, respectively, and $\sigma_1$ and $\tau_1$ are another option. The requirement that these two options are incompatible is meant to witness injectivity of $f$.

Next, player 2 picks either $\ip{\sigma_0}{\tau_0}$ or $\ip{\sigma_1}{\tau_1}$. Suppose player 2 chooses $\ip{\sigma_1}{\tau_1}$. Player 1 then plays two more pairs of strings $\ip{\sigma_0'}{\tau_0'}$ and $\ip{\sigma_1'}{\tau_1'}$ such that
\begin{enumerate}
\item $\sigma_0'$ and $\sigma_1'$ both extend $\sigma_1$ and $\tau_0', \tau_1'$ both extend $\tau_1$,
\item $\sigma_0', \sigma_1'$ are incompatible,
\item and $\tau_0', \tau_1'$ are incompatible.
\end{enumerate}
and player 2 once again picks either $\ip{\sigma_0'}{\tau_0'}$ or $\ip{\sigma_1'}{\tau_1'}$.

More generally, on each turn, player 1 plays two pairs of strings which both extend the pair of strings player 2 chose on the previous turn, making sure that the two pairs are incompatible with each other. At the end of this game, the two players have together determined two sequences, $x$ and $y$: $x$ formed from the $\sigma$'s of the pairs chosen by player 2 and $y$ formed from the $\tau$'s of the pairs chosen by player 2. Player 1 wins if $f(x) = y$.

\medskip\noindent\textbf{Formal description of the game.}
On turn $n$, player 1 plays two pairs of strings, $\ip{\sigma^n_0}{\tau^n_0}$ and $\ip{\sigma^n_1}{\tau^n_1}$, and then player 2 plays a bit $b_n \in \{0,1\}$.
\begin{align*}
\begin{array}{c | c c c c c c c c}
  \text{player } 1 & \ip{\sigma^0_0}{\tau^0_0}, \ip{\sigma^0_1}{\tau^0_1} &      & \ip{\sigma^1_0}{\tau^1_0}, \ip{\sigma^1_1}{\tau^1_1} &     & \ldots & \ip{\sigma^n_0}{\tau^n_0}, \ip{\sigma^n_1}{\tau^n_1} &      & \ldots\\
  \cline{1-9}
  \text{player } 2 &                                                     & b_0 &                                                     & b_1 & \ldots &                                                    & b_n & \ldots
\end{array}
\end{align*}
Additionally, player 1's plays must satisfy:
\begin{enumerate}
\item If $n > 1$ then $\sigma^{n}_0, \sigma^n_1$ both extend $\sigma^{n - 1}_{b_{n - 1}}$ and $\tau^n_0, \tau^n_1$ both extend $\tau^{n - 1}_{b_{n - 1}}$.
\item $\sigma^n_0, \sigma^n_1$ are incompatible
\item and $\tau^n_0, \tau^n_1$ are incompatible.
\end{enumerate}

\medskip\noindent\textbf{Winning condition.}
Let $x = \bigcup_{i \in \N}\sigma^i_{b_i}$ and $y = \bigcup_{i \in \N}\tau^i_{b_i}$. Player 1 wins if and only if $f(x) = y$.

\medskip\noindent\textbf{Case 1: player 1 wins.}
First suppose that player 1 has a winning strategy, $\gamma$. It can easily be seen that the set of plays where player 1 plays according to $\gamma$ is a perfect set on which $f$ is continuous and injective.

\medskip\noindent\textbf{Case 2: player 2 wins.}
Now suppose that player 2 has a winning strategy, $\eta$. We will show that $\range(f)$ is countable. The idea is that we can tag each element of $\range(f)$ by a unique position in the game. Since there are only countably many positions in the game, this is sufficient. Suppose $y \in \range(f)$ and $p$ is a position in the game such that each player has made exactly $n + 1$ moves so far. Say that $p$ \term{avoids} $y$ if there is some $x \in \Cantor$ such that the following hold:
\begin{enumerate}
\item $f(x) = y$
\item all moves by player 2 so far have been following $\eta$
\item all moves so far are consistent with $x, y$: for all $i \leq n$, $\sigma^i_{b_i}$ is an initial segment of $x$ and $\tau^i_{b_i}$ is an initial segment of $y$
\item and no matter what player 1 plays next, $\eta$ will not pick a move consistent with $x, y$: for all initial segments $\sigma$ of $x$ extending $\sigma^n_{b_n}$ and $\tau$ of $y$ extending $\tau^n_{b_n}$ and all $\sigma', \tau'$ such that $\ip{\sigma}{\tau}, \ip{\sigma'}{\tau'}$ is a valid next move for player 1, $\eta$ will choose $\ip{\sigma'}{\tau'}$ when given $\ip{\sigma}{\tau}, \ip{\sigma'}{\tau'}$ and when given $\ip{\sigma'}{\tau'}, \ip{\sigma}{\tau}$.
\end{enumerate}
To finish the proof, it is enough to show that any $y \in \range(f)$ is avoided by some position $p$ and that no distinct $y, y' \in \range(f)$ can be avoided by the same position.

First let's show that every $y \in \range(f)$ is avoided by some position. The idea is that if this is not the case, then we can find a way to defeat $\eta$. Fix $y \in \range(f)$ and $x \in \Cantor$ such that $f(x) = y$. Suppose $y$ is not avoided at any position. Then, in particular, $y$ is not avoided by the starting position of the game. Thus there are some initial segments $\sigma$ of $x$ and $\tau$ of $y$ and strings $\sigma', \tau'$ such that for at least one of the two moves
\begin{itemize}
\item $\ip{\sigma}{\tau}, \ip{\sigma'}{\tau'}$
\item $\ip{\sigma'}{\tau'}, \ip{\sigma}{\tau}$
\end{itemize}
by player 1, $\eta$ will choose $\ip{\sigma}{\tau}$. Consider playing this move for player 1. We are now at a position in the game where each player has made one move, player 2 has played according to $\eta$ and all moves so far are consistent with $x, y$. By assumption, $y$ is not avoided by this position either. So we can again find a move by player 1 so that $\eta$ will still choose strings consistent with $x, y$. We can continue this argument inductively to see that there is an infinite play where player 2 always plays according to $\eta$ and all moves are consistent with $x, y$. Hence the sequences formed at the end of this play are $x$ and $y$ themselves and since $x$ was chosen so that $f(x) = y$, this means player 1 wins. But this contradicts the assumption that $\eta$ is a winning strategy.

Now let's show that no distinct $y, y' \in \range(f)$ can be avoided by the same position. Suppose not. In particular, suppose $p$ is a position where each player has played $n$ moves so far and $p$ avoids both $y$ and $y'$, as witnessed by $x$ and $x'$, respectively. Since $y \neq y'$ (and consequently $x \neq x'$), there are incompatible initial segments $\tau$ of $y$ and $\tau'$ of $y'$ and incompatible initial segments $\sigma$ of $x$ and $\sigma'$ of $x'$ such that $\ip{\sigma}{\tau}$ and $\ip{\sigma'}{\tau'}$ both extend the last move of $p$. Now consider playing $\ip{\sigma}{\tau}, \ip{\sigma'}{\tau'}$ as player 1's next move in the game after $p$. Since $p$ avoids $y$ as witnessed by $x$, $\eta$ cannot choose $\ip{\sigma}{\tau}$. But since $p$ also avoids $y'$, as witnessed by $x'$, $\eta$ cannot choose $\ip{\sigma'}{\tau'}$ either. But $\eta$ must choose one of these two options, so this is a contradiction.
\end{proof}

\subsubsection*{Kihara's proof}

We will now explain how to prove part 1 of Martin's Conjecture for Borel order-preserving functions in $\ZF$, following Kihara's idea. As discussed above, Kihara's original proof used the Solecki dichotomy, but we will instead use the baby Solecki dichotomy. 

\begin{theorem}
Part 1 of Martin's Conjecture holds for all order-preserving Borel functions.
\end{theorem}

\begin{proof}
Since $f$ is order-preserving, by Theorem~\ref{thm:op_mp} it is either constant on a cone or measure-preserving. If it is constant on a cone then we are done, so we may assume it is measure-preserving. By the baby Solecki dichotomy, either $\range(f)$ is countable or $\id \leq_c f$. We will show that in the former case, $f$ is constant on a cone and in the latter case, $f$ is above the identity on a cone.
  
\medskip\noindent
\textbf{Case 1: $\range(f)$ is countable.} In this case, we can write $\Cantor$ as a countable union of Borel sets such that $f$ is constant on each one. Since there are only countably many of these sets, one of them must contain the set of paths through some pointed perfect tree, $T$. Thus $f$ is constant on $[T]$ and so $f$ is constant on a cone.

\medskip\noindent
\textbf{Case 2: $\id \leq_c f$.} By definition, there are partial continuous functions $\phi$ and $\psi$ such that for all $x \in 2^\omega$, $\psi(f(\phi(x))) = x$. Since every partial continuous function is (partial) computable relative to some oracle, we can pick some $a$ and $b$ such that $\phi$ is computable relative to $a$ and $\psi$ is computable relative to $b$.

Now consider any $x$ in the cone above $a$. Note that for such an $x$, $\phi(x) \leq_T x$. Since $\psi(f(\phi(x))) = x$, $x$ can compute some $y$ (namely $\phi(x)$) such that $f(y) \oplus b$ can compute $x$ (via $\psi$). This seems pretty close to saying that $f(x)$ can compute $x$, and hence that $f$ is above the identity, but there are a couple problems.
\begin{enumerate}
\item We are using $f(y)$ rather than $f(x)$ to compute $x$.
\item To compute $x$ from $f(y)$ we also need to know $b$, but we would like to show that $f(x)$ can compute $x$ without any extra information.
\end{enumerate}
The solution to the first problem is to note that $f$ is order-preserving and $y$ is computable from $x$ and thus $f(y)$ is computable from $f(x)$. The solution to the second problem is to use the fact that $f$ is measure-preserving, so on a high enough cone, $f(x)$ computes $b$.

Let's put all of this more formally. Let $x$ be large enough that $x$ computes $a$ and $f(x)$ computes $b$. Thus $\phi(x)$ is computable from $x$ and since $f$ is order-preserving, this means that $f(x)$ computes $f(\phi(x))$. Since $f(x)$ also computes $b$, $f(x)$ computes $\psi(f(\phi(x))) = x$. So $f(x)$ computes $x$ on a cone.
\end{proof}

\begin{remark}
We will now verify that when $f$ is Borel, the above proof can be carried out in $\ZF$ alone. The proof begins by invoking Theorem~\ref{thm:op_mp}, but as noted in Remark~\ref{rmk:op_mp_borel}, if $f$ is Borel then this theorem is provable in $\ZF$. The proof then uses the baby Solecki dichotomy. The proof we gave above of this dichotomy only uses Borel determinacy and thus works in $\ZF$. The proof of case 1 implicitly uses determinacy in the form of Corollary~\ref{cor:framework_pointed}, but as noted in Remark~\ref{rmk:op_mp_borel}, when $f$ is Borel this only requires Borel determinacy. Finally, the proof of case 2 does not use determinacy at all.
\end{remark}

\subsection{Application to the theory of locally countable Borel quasi-orders}
\label{sec:op_application}
We will now discuss an application of our result on order-preserving functions to the theory of \term{locally countable Borel quasi-orders}. In particular, we will show that Turing reducibility is not a universal locally countable Borel equivalence relation.

The motivation for this result comes from the theory of countable Borel equivalence relations (for a survey of this theory, see~\cite{kechris2021theory}). Kechris has conjectured that Turing equivalence is a universal countable Borel equivalence relation, which is known to contradict Martin's Conjecture~\cite{dougherty2000how}. The result we prove in this section refutes a natural strengthening of Kechris's conjecture.

Before proving our result, we begin by reviewing some definitions.

\begin{definition}
A quasi-order $\leq_X$ on a Polish space $X$ is \term{Borel} if it is Borel as a subset of $X\times X$ and \term{locally countable} if for every element $x$ of $X$, the set of predecessors of $x$---i.e.\ $\{y \mid y \leq_X x\}$---is countable.
\end{definition}

Note that Turing reducibility on $\Cantor$ is a locally countable Borel quasi-order.

\begin{definition}
If $\leq_X$ is a Borel quasi-order on $X$ and $\leq_Y$ is a Borel quasi-order on $Y$ then $\leq_X$ is \term{Borel reducible} to $\leq_Y$ if there is a Borel function $f\colon X \to Y$ such that for all $x_1, x_2 \in X$
\[
  x_1 \leq_X x_2 \iff f(x_1) \leq_X f(x_2).
\]
\end{definition}

\begin{definition}
A locally countable Borel quasi-order is \term{universal} if every other locally countable Borel quasi-order is Borel reducible to it.
\end{definition}

To show that Turing reducibility is not universal, we just need to exhibit a single locally countable Borel quasi-order that is not reducible to it (the general question of which locally countable Borel quasi-orders are reducible to Turing reducibility is explored more thoroughly in \cite{higuchi2023note}). We will show that if we take Turing reducibility and add one extra point which is not comparable with anything else then the resulting quasi-order is not reducible to Turing reducibility.

Here's the main idea of the proof. If we have a Borel reduction from Turing reducibility plus a point to regular Turing reducibility then by ignoring the extra point, we get an injective, order-preserving Borel function on the Turing degrees. By the Borel version of the results of section \ref{sec:op_mp}, this function must be measure-preserving. But this means that this function eventually gets above the image of the extra point, which contradicts the fact that it is a reduction (since the extra point is supposed to be incomparable to everything else).

More formally, we can define a quasi-order as follows. Let $0$ denote the element of $2^\omega$ whose bits are all $0$s and let $\le_T^*$ be the binary relation on $2^\omega$ defined as follows.
\[
  x \le_T^* y \iff
  \begin{cases}
    x \le_T y \text{ and } x, y \neq 0\\
    \text{ or } x = y = 0.
  \end{cases}
\]
In other words, $\le_T^*$ is exactly like Turing reducibility except that there is a special point, $0$, which is not comparable to anything else. It is easy to see that this is a locally countable Borel quasi order.

\begin{theorem}
The quasi order $\le_T^*$ is not Borel reducible to $\le_T$.
\end{theorem}

\begin{proof}
Suppose for contradiction that $f\colon 2^\omega \to 2^\omega$ is a Borel reduction from $\le_T^*$ to $\le_T$. Let $f^*$ denote the function $f$ restricted to all the reals not equal to $0$. By definition of ``Borel reduction,'' $f^*$ is a Borel order-preserving function which is injective on the Turing degrees (though not necessarily on the reals). By Theorem \ref{thm:op_mp}\footnote{Note that that theorem was stated as a theorem of $\ZF + \AD$, but when restricted to Borel functions, it is provable in $\ZF$. The key point is that the perfect set theorem used in the proof holds for all analytic sets in $\ZF$.}, either $f^*$ is constant on a cone or measure-preserving. But since it is injective on the Turing degrees it cannot be constant on a cone and thus must be measure-preserving.

Since $f^*$ is measure-preserving, its range is cofinal in the Turing degrees. Thus there is some $x\neq 0$ such that $f(0) \le_T f^*(x)$. But $f^*(x)$ is just $f(x)$ and hence we have $f(0) \le_T f(x)$. Since $f$ is a Borel reduction, this implies that $0 \le_T^* x$, which contradicts the definition of $\le_T^*$ (since $0$ is supposed to be incomparable to all other elements).
\end{proof}

\begin{corollary}
Turing reducibility is not a universal locally countable Borel quasi order.
\end{corollary}

\section{Ultrafilters on the Turing Degrees}
\label{sec:ultra}

Several of the definitions, facts and theorems related to Martin's Conjecture can be recast in the language of ultrafilters. When recast in this language, our result on measure-preserving functions allows us to show that part 1 of Martin's Conjecture is equivalent to a statement about the structure of the class of ultrafilters on the Turing degrees. This equivalent statement suggests a few routes to making progress on Martin's Conjecture, one of which we will explore further.

Central to this view of Martin's Conjecture is the \term{Martin measure}, also known as the \term{cone measure}, a countably complete filter on the Turing degrees.

\begin{definition}
The \term{Martin measure} on the Turing degrees is the class, $U_M$, of all sets of Turing degrees which contain a cone, i.e.\
\[
  U_M = \{A \subseteq \D_T \mid \text{ for some $x$, } \Cone(x) \subseteq A\}.
\]
\end{definition}

The most important fact about the Martin measure is that, under $\AD$, it is an ultrafilter.

\begin{theorem}[$\ZF + \AD$; Martin]
The Martin measure is an ultrafilter on the Turing degrees.
\end{theorem}

This theorem is simply a restatement of Martin's cone theorem in terms of the Martin measure. Likewise, several key definitions can be restated in terms of Martin measure. For example, the Martin order and Martin equivalence can be defined as follows.
\begin{itemize}
\item $f \leq_M g$ if and only if $F(\degree{x}) \leq_T G(\degree{x})$ for $U_M$-almost every $\degree{x}$ (where $F$ and $G$ are the functions on the Turing degrees induced by $f$ and $g$, respectively).
\item $f \equiv_M g$ if and only if $F(\degree{x}) = G(\degree{x})$ for $U_M$-almost every $\degree{x}$.
\end{itemize}
We will now see that the class of measure-preserving functions on the Turing degrees also has a natural definition in terms of the Martin measure: it is exactly the class of functions which are measure-preserving for the Martin measure in the sense of ergodic theory (which is the reason that we chose to call them ``measure-preserving''). To explain this, we first need to recall some definitions from measure theory.

\begin{remark}
Several of the results in this section are proved in the theory $\ZF + \AD + \Unif_\R$. All of these results can also be proved in the theory $\ZF + \AD^+$. We don't know whether they can be proved in $\ZF + \AD + \DC_\R$.
\end{remark}

\subsection{Measure-preserving functions and the Martin measure}
\label{sec:ultra_def}
\subsubsection*{Measure-preserving functions in general}

Given a measure $\mu$ on a set $X$ and a function $f \colon X \to Y$, there is a canonical way of getting a measure on $Y$, called the pushforward of $\mu$ by $f$.

\begin{definition}
If $\mu$ is a measure on a space $X$ and $f \colon X \to Y$ is a function then the \term{pushforward} of $\mu$ by $f$, denoted $f_*\mu$, is the measure on $Y$ given by
\[
f_*\mu(A) = \mu(f^{-1}(A)).
\]
\end{definition}

\begin{proposition}
The pushforward of a measure is itself a measure.
\end{proposition}

Note that for any functions $f\colon X \to Y$ and $g \colon Y \to Z$, $(g\circ f)_* \mu = g_* f_* \mu$. We can now define what it means for a function to be measure-preserving (in the sense of ergodic theory).

\begin{definition}
If $\mu$ is a measure on a space $X$ and $f \colon X \to X$ is a function then $f$ is \term{measure-preserving} for $\mu$, or sometimes is said to \term{preserve} $\mu$, if $f_*\mu = \mu$.
\end{definition}

We will be mostly concerned with measures that are ultrafilters on the Turing degrees. Recall that an ultrafilter on a set $X$ can be considered a $\{0, 1\}$-valued measure on $X$. The next proposition tells us that we can talk about the pushforwards of ultrafilters without worrying about measures which are not ultrafilters.

\begin{proposition}
The pushforward of an ultrafilter is itself an ultrafilter.
\end{proposition}

Suppose that $U$ is an ultrafilter on a set $X$, $V$ is an ultrafilter on a set $Y$, $f \colon X \to Y$ is any function and we would like to determine whether $f_* U = V$. If we simply use the definition of pushforward, we must check that for each $A \in V$, $f^{-1}(A) \in U$ \emph{and} that for each $A \notin V$, $f^{-1}(A) \notin U$. The following lemma tells us that because $U$ and $V$ are ultrafilters, this condition can be simplified somewhat.

\begin{lemma}
\label{lemma:ultra_mpcondition}
If $U$ is an ultrafilter on $X$, $V$ is an ultrafilter on $Y$ and $f \colon X \to Y$ then $f_*U = V$ if and only if for all $A \in U$, $f(A) \in V$.
\end{lemma}

\begin{proof}
($\implies$) First, suppose that $f_*U = V$ and let $A$ be any set in $U$. By the definition of pushforward, $f(A)$ is in $V$ if and only if $f^{-1}(f(A))$ is in $U$. Since $f^{-1}(f(A))$ clearly contains $A$ and $A$ is in $U$, we can conclude that $f(A)$ is in $V$.

\medskip\noindent
($\impliedby$) Now suppose that for all $A \in U$, $f(A) \in V$. We need to show that for each $B \subseteq Y$, $B \in V$ if and only if $f^{-1}(B) \in U$.

First assume $f^{-1}(B)$ is in $U$. Then by our assumption, $f(f^{-1}(B))$ is in $V$ and therefore so is $B$ (since it is a superset of $f(f^{-1}(B))$).

Now assume that $f^{-1}(B)$ is not in $U$. Since $U$ is an ultrafilter, this means $(f^{-1}(B))^C = f^{-1}(B^C)$ \emph{is} in $U$ (where we use $(\cdot)^C$ to denote the appropriate relative complement of a set). By the reasoning in the preceding paragraph, this means that $B^C$ is in $V$ and hence that $B$ is not.
\end{proof}

\subsubsection*{Measure-preserving $=$ measure-preserving}

We can now give our equivalent definition of the class of measure-preserving function on the Turing degrees in terms of the Martin measure.

\begin{proposition}
A Turing-invariant function $f \colon 2^\omega \to 2^\omega$ is measure-preserving if and only if the function $F\colon \D_T \to \D_T$ that it induces is measure-preserving for the Martin measure.
\end{proposition}

\begin{proof}($\implies$) Suppose $f$ is measure-preserving (in the sense of Definition \ref{def:intro_mp}). By Lemma \ref{lemma:ultra_mpcondition}, we just need to show that if a subset $A$ of the Turing degrees contains a cone then $F(A)$ contains a cone. By determinacy, it is enough to show that $F(A)$ is cofinal in the Turing degrees. So let $\degree{a}$ be any Turing degree and we need to show that there is some degree above $\degree{a}$ which is in $F(A)$. Since $f$ is measure preserving, there is some $\degree{b}$ such that for all $\degree{x}$, $\degree{x} \geq_T \degree{b}$ implies $F(\degree{x}) \geq_T \degree{a}$. And since $A$ contains a cone, we can find some $\degree{x} \in A$ such that $\degree{x} \geq_T \degree{b}$ and hence $F(\degree{x}) \geq_T \degree{a}$.

\medskip\noindent
($\impliedby$) Suppose $F$ is measure-preserving for the Martin measure. Let $a$ be any real. We need to find some real so that on the cone above that real, $f$ is always above $a$. Since $F$ preserves the Martin measure, $F^{-1}(\Cone(a))$ must contain a cone. Let $b$ be a base of such a cone. Then for any $x \ge_T b$, $F(\degreeof(x)) \in \Cone(a)$ and hence $f(x) \ge_T a$.
\end{proof}

\subsection{The Rudin-Keisler order on ultrafilters on the Turing degrees}
\label{sec:ultra_rk}
In the previous section we saw that measure-preserving functions can be defined in terms of the Martin measure. In this section we will see that our result on part 1 of Martin's Conjecture for measure-preserving functions implies that, at least under $\AD + \Unif_\R$, part 1 of Martin's Conjecture is equivalent to a statement about ultrafilters on the Turing degrees.

\subsubsection*{The Rudin-Keisler Order}

To explain the connection between part 1 of Martin's Conjecture and ultrafilters on the Turing degrees, we first need to give some background on the \term{Rudin-Keisler order} on ultrafilters.

\begin{definition}
Suppose $U$ is an ultrafilter on a set $X$ and $V$ is an ultrafilter on a set $Y$. Then $U$ is \term{Rudin-Keisler below} $V$, written $U \le_{RK} V$, if there is a function $f \colon Y \to X$ such that
\[
  f_*V = U.
\]
\end{definition}

\begin{example}
If $U$ is a principal ultrafilter on a set $X$ then $U$ is Rudin-Keisler below every other ultrafilter. To see why, suppose $U$ concentrates on the point $a \in X$ and suppose $V$ is an ultrafilter on a set $Y$. It is easy to check that if $f\colon Y \to X$ is the constant function $x\mapsto a$ then $f_*(V) = U$.
\end{example}

Note that in the definition of $\le_{RK}$, the function $f$ is going in the opposite direction from what one might naively expect. This makes more sense if one considers embeddings of ultrapowers: if $U \le_{RK} V$ then for every structure $M$ there is an embedding $M^X/U \to M^Y/V$. Also note that it is possible to have distinct ultrafilters $U$ and $V$ such that $U \le_{RK} V$ and $V \le_{RK} U$; in other words, $\le_{RK}$ is only a quasi-order rather than a partial order. In case this happens we will say that $U$ and $V$ are \term{weakly Rudin-Keisler equivalent}.

\begin{definition}
Suppose $U$ is an ultrafilter on a set $X$ and $V$ is an ultrafilter on a set $Y$. Then $U$ is \term{weakly Rudin-Keisler equivalent} to $V$, written $U \equiv_{RK} V$, if $U \le_{RK} V$ and $V \le_{RK} U$.
\end{definition}

Note that this definition is slightly different than the usual definition of ``Rudin-Keisler equivalent'' found in the literature (which is why we have added the word ``weakly''). The usual definition is that ultrafilters $U$ on $X$ and $V$ on $Y$ are Rudin-Keisler equivalent if there is a bijection $f \colon X \to Y$ such that $f_* U = V$. The two definitions are equivalent under $\ZFC$, but not under $\ZF$.

If we restrict our attention to a ultrafilters on a single set and ignore the principal ultrafilters, then the class of ultrafilters which are minimal in the Rudin-Keisler order often turns out to be an important class with a natural characterization that does not mention the Rudin-Keisler order. We must be slightly careful here about what we mean by ``minimal.'' We mean minimal in the sense of a quasi-order, i.e.\ $U$ is minimal if for all $V \leq_{RK} U$, $V$ is weakly Rudin-Keisler equivalent to $U$.

\begin{example}
The minimal nonprincipal ultrafilters on $\omega$ are exactly the Ramsey ultrafilters.
\end{example}

\begin{example}
Every normal ultrafilter on a cardinal $\kappa$ is Rudin-Keisler minimal among nonprincipal ultrafilters on $\kappa$ and every minimal nonprincipal ultrafilter on $\kappa$ is Rudin-Keisler equivalent to either a normal ultrafilter or to a Ramsey ultrafilter on $\omega$ (see chapter 9 of~\cite{comfort1974theory}).
\end{example}

\subsubsection*{Part 1 of Martin's Conjecture and the Rudin-Keisler Order}

We will now show that, at least under $\AD + \Unif_\R$ (or $\AD^+$), part 1 of Martin's Conjecture is equivalent to a statement about the position of the Martin measure in the Rudin-Keisler order on nonprincipal ultrafilters on the Turing degrees.

\begin{theorem}[$\ZF + \AD + \Unif_\R$]
\label{thm:ultra_rk}
Part 1 of Martin's Conjecture is equivalent to the following statement: if $V$ is a nonprincipal ultrafilter on $\D_T$ such that $V \le_{RK} U_M$ then $V = U_M$.
\end{theorem}

Note that this is a bit stronger than saying that $U_M$ is minimal among the nonprincipal ultrafilters on the Turing degrees (though it does imply that). In particular, this statement rules out the existence of a nonprincipal ultrafilter $V$ on $\D_T$ which is weakly Rudin-Keisler equivalent, but not literally equal, to $U_M$.

\begin{proof}
Since we are working under $\Unif_\R$, part 1 of Martin's Conjecture is equivalent to the statement that every function $F \colon \D_T \to \D_T$ is either constant on a cone or above the identity on a cone (the point is that under $\Unif_\R$ every function on the Turing degrees is induced by a Turing-invariant function on the reals). We have the following equivalences.
\begin{itemize}
\item $F$ is constant on a cone if and only if $F_*U_M$ is a principal ultrafilter.
\item $F$ is above the identity on a cone if and only if $F$ is measure-preserving (one direction is clear from the definitions and the other follows from Theorem~\ref{thm:mp_part1_adr}). By the equivalent definition of measure-preserving in terms of Martin measure, this means $F$ is above the identity on a cone if and only if $F_*U_M = U_M$.
\end{itemize}
Thus $F$ is either constant on a cone or above the identity on a cone if and only if $F_*U_M$ is either a principal ultrafilter or $U_M$ itself. By definition of $\leq_{RK}$, the latter statement holds for all $F$ if and only if no nonprincipal ultrafilter on $\D_T$ besides $U_M$ itself is below $U_M$ in the Rudin-Keisler order.
\end{proof}

This equivalence suggests a few approaches to part 1 of Martin's Conjecture.
\begin{itemize}
\item \textbf{Use work on the Rudin-Keisler order from set theory.} For example, we have mentioned that a normal ultrafilter on a cardinal $\kappa$ is $\leq_{RK}$-minimal among nonprincipal ultrafilters on $\kappa$. Notably, Slaman and Steel's theorem on Martin's Conjecture for regressive functions can be seen as providing a kind of analogue of normality for Martin measure.
\item \textbf{Study specific ultrafilters or classes of ultrafilters.} Given a specific ultrafilter (or class of ultrafilters) on the Turing degrees, one could try to show that this ultrafilter is not Rudin-Keisler below the Martin measure. We will discuss this approach more in the next section.
\item \textbf{Split part 1 of Martin's Conjecture into two parts.} By Theorem~\ref{thm:ultra_rk}, to prove part 1 of Martin's Conjecture it is enough to prove two things: if $V$ is a nonprincipal ultrafilter on $\D_T$ such that $V \leq_{RK} U_M$ then $U_M \leq_{RK} V$ (i.e.\ they are weakly equivalent) and if $U_M \leq_{RK} V \leq_{RK} U_M$ then $V = U_M$. Perhaps one of these is easier to prove on its own than the full part 1 of Martin's Conjecture. We will discuss a proposition relevant to the latter of these two parts in section~\ref{sec:ultra_extra}.
\end{itemize}

\subsection{The Lebesgue and Baire ultrafilters}
\label{sec:ultra_lebesgue}
It is possible to show that under $\AD$, Lebesgue measure on $\Cantor$ induces an ultrafilter on the Turing degrees. Likewise, under $\AD$, the Baire filter (i.e.\ the class of comeager sets) on $\Cantor$ induces an ultrafilter on the Turing degrees. In light of the discussion in the previous section, it would be interesting to show that these two ultrafilters are not below the Martin measure in the Rudin-Keisler order. We don't know how to do that but we can show that neither of them is \emph{above} the Martin measure in the Rudin-Keisler order. This might sound like the wrong direction, but it's not as bad as it sounds: it shows that neither ultrafilter is weakly equivalent to the Martin measure. 

Note that Andrew Marks and Adam Day independently proved the results in this section several years before we did and also asked whether the Lebesgue and Baire ultrafilters are below the Martin measure in the Rudin-Keisler order (as well as other questions about the Rudin-Keisler order on ultrafilters on the Turing degrees)~\cite{marks2020personal}. 

\subsubsection*{Lebesgue and Baire are ultrafilters}

First, we define the Lebesgue and Baire filters on $\D_T$ as follows.
\begin{itemize}
\item \textbf{The Lebesgue filter:} let $U_L$ denote the class of subsets of $\D_T$ with measure $1$, i.e.\
  \[
    A \in U_L \iff \lambda(\{x \in \Cantor \mid \degreeof(x) \in A\}) = 1.
  \]
where $\lambda$ denotes the Lebesgue measure on $\Cantor$.
\item \textbf{The Baire filter:} let $U_B$ denote the class of subsets of $\D_T$ which are comeager, i.e.\
  \[
    A \in U_B \iff \{x \in \Cantor \mid \degreeof(x) \in A\} \text{ is comeager in } \Cantor.
  \]
\end{itemize}
Note that both $U_L$ and $U_B$ are countably complete filters on $\D_T$.

As we mentioned above, $\AD$ implies that both $U_L$ and $U_B$ are actually ultrafilters. This can be proved using Kolmogorov's zero-one law together with standard regularity properties implied by $\AD$.

\begin{definition}
A set $A \subseteq 2^\omega$ is \term{closed under tail equivalence} if for all $x, y \in 2^\omega$ which differ at only finitely many positions
\[
  x \in A \iff y \in A.
\]
\end{definition}

\begin{theorem}[Kolmogorov's zero-one law]
Suppose $A \subseteq 2^\omega$ is closed under tail equivalence.
\begin{enumerate}
\item If $A$ is Lebesgue measurable then either $\lambda(A) = 0$ or $\lambda(A) = 1$.
\item If $A$ has the Baire property (i.e.\ $A$ is either meager or comeager in some basic open set) then $A$ is either meager or comeager.
\end{enumerate}
\end{theorem}

\begin{theorem}[$\ZF + \AD$]
Every subset of $2^\omega$ is Lebesgue measurable and has the Baire property.
\end{theorem}

\begin{proposition}[$\ZF + \AD$]
$U_L$ and $U_B$ are ultrafilters.
\end{proposition}

\begin{proof}
We will just provide a proof for $U_L$ since the proof for $U_B$ is almost identical. Let $A$ be any set of Turing degrees. We want to show that either $A$ is in $U_L$ or $\D_T\setminus A$ is in $U_L$. Since Turing degrees are closed under tail equivalence, the set $\{x \mid \degreeof(x) \in A\}$ is also closed under tail equivalence. $\AD$ implies that it is Lebesgue measurable and thus by Kolmogorov's zero-one law it either has measure $0$ or measure $1$. In the former case, $\D_T \setminus A \in U_L$. In the latter case, $A \in U_L$.
\end{proof}

\subsubsection*{Lebesgue and Baire are not above Martin}

We will now prove that the Lebesgue and Baire ultrafilters are not Rudin-Keisler above the Martin measure. Our strategy is as follows. Suppose $F \colon \D_T \to \D_T$ is a function such that $F_* U_L = U_M$. By composing $F$ with the map $\degree{x} \mapsto \omega_1^{\degree{x}}$ and then taking the pushforward of $U_L$ by this function, we get a nonprincipal countably complete ultrafilter on $\omega_1$. However, countably complete ultrafilters on $\omega_1$ are rather constrained. In particular, they can only satisfy Fubini's theorem if they are principal. Since Lebesgue measure does satisfy Fubini's theorem, we can use this to derive a contradiction.

\begin{lemma}[$\ZF + \AD$]
\label{lemma:ultra_constantordinal}
Every function $f \colon 2^\omega \to \omega_1$ is constant on a set of positive Lebesgue measure.
\end{lemma}

\begin{proof}
Suppose for contradiction that $f$ is not constant on any set of positive measure. Note that for every $\alpha \in \omega_1$, $\AD$ implies that $f^{-1}(\alpha)$ is Lebesgue measurable and so our assumption implies that it has measure $0$. By countable additivity of the Lebesgue measure, this implies that for any countable set $A \subseteq \omega_1$, $f^{-1}(A)$ has measure $0$.

Now let $B$ be the subset of $2^\omega\times 2^\omega$ defined by
\[
B = \{(x, y) \mid f(x) \le f(y)\}.
\]
Again, since we are working under $\AD$, we know $B$ is Lebesgue measurable. We will now use Fubini's theorem to compute the measure of $B$ in two different ways to arrive at a contradiction. By Fubini's theorem we have:
\[
  \lambda(B) = \int \lambda(\{y \mid (x, y) \in B\})\,dx = \int\lambda(\{x \mid (x, y) \in B\})\,dy.
\]
Now note that for any $y$, we have
\[
  \{x \mid (x, y) \in B\} = f^{-1}(\{\alpha \mid \alpha \le f(y)\}).
\]
Since this is the inverse image under $f$ of a countable set, its measure must be $0$. Thus
\[
  \int\lambda(\{x \mid (x, y) \in B\})\,dy = \int 0\, dy = 0.
\]
On the other hand, for any $x$,
\[
  \{y \mid (x, y) \in B\} = f^{-1}(\{\alpha \mid f(x) \le \alpha\}).
\]
Since this is the complement of the inverse image under $f$ of a countable set, its measure must be $1$. Thus
\[
  \int\lambda(\{y \mid (x, y) \in B\})\,dx = \int 1\, dx = 1.
\]
Therefore we have calculated that the measure of $B$ is both $0$ and $1$, a contradiction.
\end{proof}

\begin{corollary}
Every function $F\colon \D_T \to \omega_1$ is constant on some set in $U_L$.
\end{corollary}

\begin{proof}
This follows from the previous lemma together with Kolmogorov's zero-one law.
\end{proof}

\begin{theorem}[$\ZF + \AD$]
\label{thm:ultra_lebesgue}
The Lebesgue ultrafilter is not Rudin-Keisler above the Martin measure, i.e.\ $U_M \nleq_{RK} U_L$.
\end{theorem}

\begin{proof}
Suppose for contradiction that $U_M \le_{RK} U_L$, as witnessed by some function $F$ (so $F_* U_L = U_M$). Let $G \colon \D_T \to \omega_1$ be the map defined by
\[
  G(\degree{x}) = \omega_1^{\degree{x}}.
\]
It is straightforward to check that $G_*(U_M)$ is a nonprincipal ultrafilter on $\omega_1$. But by the lemma above, $G\circ F$ is constant on a set of Lebesgue measure 1 and hence $(G\circ F)_*(U_L)$ is a principal ultrafilter on $\omega_1$. This is a contradiction since we have
\[
  (G\circ F)_*(U_L) = G_*(F_*(U_L)) = G_*(U_M). \qedhere
\]
\end{proof}

\begin{theorem}[$\ZF + \AD$]
\label{thm:ultra_baire}
The Baire ultrafilter is not Rudin-Keisler above the Martin measure, i.e.\ $U_M \nleq_{RK} U_B$.
\end{theorem}

\begin{proof}
We can repeat the proof for the Lebesgue ultrafilter almost verbatim. In particular, Baire category satisfies a version of Fubini's theorem which is sufficient to carry out Lemma~\ref{lemma:ultra_constantordinal}.
\end{proof}

In light of the results above, it would be very interesting to show that the Lebesgue and Baire ultrafilters are not below the Martin measure in the Rudin-Keisler order. Note that Andrew Marks and Adam Day have shown that under $\AD_\R$, $U_L \leq_{RK} U_M$ holds if and only if there is a Turing-invariant function $f \colon \Cantor \to \Cantor$ such that for all $x$, $f(x)$ is $x$-random~\cite{marks2020personal}.

Our results also suggest some questions which are less directly connected to Martin's Conjecture. For example, we showed that $U_L$ and $U_B$ are not Rudin-Keisler above the Martin measure. Is there any ultrafilter on $\D_T$ which is Rudin-Keisler above the Martin measure (besides the Martin measure itself)? Also, is there any meaningful difference between the Lebesgue and Baire ultrafilters on $\D_T$ with respect to the Rudin-Keisler order?

\subsection{Additional facts about the Martin measure and the Rudin-Keisler order}
\label{sec:ultra_extra}
We will finish our discussion of Martin's Conjecture and ultrafilters on the Turing degrees by mentioning a few facts which seem relevant to attempts to investigate the place of the Martin measure in the Rudin-Keisler order.

\subsubsection*{A characterization of the Martin measure}

In order to use the approach to part 1 of Martin's Conjecture described in Section~\ref{sec:ultra_rk}, it is necessary to have tools to show that ultrafilters on the Turing degrees are equal to the Martin measure. We do not know of many such tools, except for the following proposition.

\begin{proposition}[$\ZF + \AD$]
\label{prop:mp_equalsmartin}
Suppose $V$ is a nonprincipal ultrafilter on the Turing degrees such that $\{\degree{x} \mid \Cone(\degree{x}) \in V\}$ is in $V$. Then $V = U_M$.
\end{proposition}

\begin{proof}
Let $A = \{\degree{x} \mid \Cone(\degree{x}) \in V\}$. We will first assume that $A$ is cofinal in the Turing degrees and show that this implies $U_M = V$. We will then show that $A$ is cofinal.

\medskip\noindent\textbf{The cofinality of $A$ is sufficient.}
Assume $A$ is cofinal in the Turing degrees. Since $U_M$ and $V$ are both ultrafilters, it is enough to show that every set in $U_M$ is in $V$. To show this, it is enough to show that every cone is in $V$. 

Fix an arbitrary degree $\degree{x}$. We will show that the cone above $\degree{x}$ is in $V$. Since $A$ is cofinal, there is some $\degree{y} \geq_T \degree{x}$ such that $\degree{y} \in A$. By definition of $A$, this implies that $\Cone(\degree{y}) \in V$ and hence that $\Cone(\degree{x})$ (which is a superset of $\Cone(\degree{y})$) is in $V$.

\medskip\noindent\textbf{$A$ is cofinal.} 
Let $\tilde{A} = \{x \in \Cantor \mid \degreeof{x} \in A\}$. The idea is to apply Corollary~\ref{cor:op_directed} to $\tilde{A}$ to show that $\tilde{A}$ (and hence $A$ as well) is cofinal. 

First, we claim that $A$ is not countable. Since we are working in $\ZF + \AD$, $V$ is countably complete (since every ultrafilter is). Since $V$ is nonprincipal and countably complete and $A$ is in $V$, $A$ cannot be countable.

Next we claim that $A$ is countably directed in the Turing degrees. Let $\degree{x}_0, \degree{x}_1, \ldots$ be a countable sequence of elements of $A$. Hence for each $n$, $\Cone(\degree{x}_n) \in V$. So $\bigcap_{n} \Cone(\degree{x}_n)$ is also in $V$, since $V$ is countably complete. By assumption, $A$ is in $V$ and hence
\[
A \cap \bigcap_{n} \Cone(\degree{x}_n)
\]
is nonempty. It is easy to check that any element of this intersection gives us an element of $A$ which is an upper bound for all the $\degree{x}_n$'s.

Since $A$ is not countable, neither is $\tilde{A}$. Hence the perfect set theorem implies that $\tilde{A}$ contains a perfect set. Since $A$ is countably directed in the Turing degrees, so is $\tilde{A}$. Thus Corollary~\ref{cor:op_directed} implies that $\tilde{A}$ is cofinal in the Turing degrees and so $A$ is as well.
\end{proof}

This proposition has an interesting consequence in the world of Turing-invariant functions, which can be seen as a sharpening of our result on order-preserving functions.

\begin{definition}
A Turing-invariant function $f \colon \Cantor \to \Cantor$ is \term{almost order-preserving} if for every $x$, there is a cone on which $f$ is above $f(x)$---i.e.\ for all $y$ in some cone, $f(y) \geq_T f(x)$.
\end{definition}

Note that if $f$ is order-preserving then $f(y) \geq_T f(x)$ for all $y \geq_T x$ whereas if $f$ is merely almost order-preserving then $f(y) \geq_T f(x)$ only for all $y$ of sufficiently high Turing degree.

\begin{corollary}[$\ZF + \AD + \DC_\R$]
Suppose $f\colon \Cantor \to \Cantor$ is almost order-preserving. Then either $f$ is constant on a cone or $f \geq_M \id$.
\end{corollary}

\begin{proof}
Assume $f$ is not constant on any cone and let $F \colon \D_T \to \D_T$ be the function on $\D_T$ induced by $f$. We will show $F_* U_M$ has the property in the statement of Proposition~\ref{prop:mp_equalsmartin}. To see why, note that $F_* U_M$ concentrates on the image of $F$ and it follows immediately from the definition of almost order-preserving that for every $\degree{x}$, the cone above $\degree{x}$ is in $F_* U_M$. Also note that since $f$ is not constant on any cone, $F_* U_M$ is nonprincipal.

By Proposition~\ref{prop:mp_equalsmartin}, this implies that $F_* U_M = U_M$. In other words, $f$ is measure-preserving. So by Theorem~\ref{thm:mp_part1measurepreserving}, $f \geq_M \id$, as desired.
\end{proof}

\subsubsection*{The Rudin-Keisler order below the Martin measure}

Earlier, we showed that part 1 of Martin's Conjecture is equivalent to the statement that there are no nonprincipal ultrafilters on the Turing degrees which are Rudin-Keisler below the Martin measure (other than the Martin measure itself). It seems reasonable to ask whether there are any nonprincipal ultrafilters at all that are Rudin-Keisler below the Martin measure, even if we consider ultrafilters on sets other than the Turing degrees. In fact, a theorem due to Kunen implies that there are many such ultrafilters: in particular, any ultrafilter on an ordinal less than $\Theta$ (where $\Theta$ is the least ordinal such that there is no surjection $\R \to \Theta$).

\begin{theorem}
    [$\ZF+\AD+\DC_\R$; Kunen]\label{thm:kunen} If $V$ is an ultrafilter on an ordinal $\kappa<\Theta$ then $V\leq_{RK} U_M$.
\end{theorem}

A proof of this theorem can be found in a paper by Steel~\cite{steel2009derived} in the course of a proof of a more well-known theorem of Kunen which states that every ultrafilter on any ordinal less than $\Theta$ is ordinal definable (Theorem 8.6 of Steel's paper).

\section{Generalizations and Counterexamples}
\label{sec:gen}

In this section we will discuss the extent to which our results can be generalized to other contexts. In particular, we will consider whether they still hold in degree structures other than the Turing degrees, whether they hold for functions which are not Turing-invariant, and whether they hold (in a modified form) in $\ZFC$.

In general, we find that part 1 of Martin's Conjecture for measure-preserving functions is fairly robust, holding in a number of different contexts, while our results on order-preserving functions are much harder to generalize.

\subsection{Other degree structures}
\label{sec:gen_degree}
There are many degree structures besides the Turing degrees studied in computability theory and it is possible to state a version of Martin's Conjecture in a number of these structures. All that's required is that the degree structure satisfy the appropriate analog of Martin's cone theorem and have a notion of a jump operator. For example, both the arithmetic degrees and the hyperarithmetic degrees satisfy these requirements---in the arithmetic degrees, the appropriate jump operator is the $\omega$-jump, while in the hyperarithmetic degrees it's the hyperjump---and thus we can state a version of Martin's Conjecture for both.

Surprisingly, these different versions of Martin's Conjecture have turned out to work somewhat differently from each other. Some of the special cases of Martin's Conjecture which are known to hold in the Turing degrees are known to be false in the arithmetic degrees. And there are other instances of Martin's Conjecture which are known to hold for the Turing degrees, but whose status is open for the arithmetic degrees or the hyperarithmetic degrees.

For example, here is the status of two special cases of Martin's Conjecture for the arithmetic and hyperarithmetic degrees.
\begin{itemize}
\item \textbf{Martin's Conjecture for uniformly invariant functions:} This fails in the arithmetic degrees~\cite{marks2016martins}, but holds in the hyperarithmetic degrees (unpublished, but mentioned in~\cite{slaman1988definable}).
\item \textbf{Martin's Conjecture for regressive functions:} This is open for the arithmetic degrees and holds for the hyperarithmetic degrees~\cite{lutz2023martins}.
\end{itemize}

In this section, we will discuss to what extent the results in this paper extend to the arithmetic and hyperarithmetic degrees. In brief, the results on measure-preserving functions seem quite robust and hold in most degree structures, while the results on order-preserving functions seem to rely more on the specific structure of the Turing degrees.

\subsubsection*{Measure-preserving functions on other degree structures}

Both of our proofs of part 1 of Martin's Conjecture for measure-preserving functions seem to be very flexible and can be made to work in most degree structures, including both the arithmetic and hyperarithmetic degrees. Rather than recapitulate our proof in detail for many different degree structures, we will just give a sketch of the key points for the case of the arithmetic degrees. We will use the $\AD + \Unif_\R$ proof since it is simpler to explain, but our second proof can also be easily adapted to the arithmetic degrees.

For the sake of completeness, we begin by stating the definition of ``measure-preserving'' for arithmetically invariant functions.

\begin{definition}
An arithmetically invariant function $f \colon 2^\omega \to 2^\omega$ is \term{measure-preserving} if for every $a$ there is some $b$ such that
\[
  x \ge_A b \implies f(x) \ge_A a.
\]
In other words, for every $a$, $f$ is arithmetically above $a$ on a cone of arithmetic degrees.
\end{definition}

\begin{theorem}[$\ZF + \AD + \Unif_\R$]
If $f \colon 2^\omega \to 2^\omega$ is an arithmetically invariant, measure-preserving function then $f(x) \ge_A x$ on a cone of arithmetic degrees.
\end{theorem}

\begin{proof}
First, use $\Unif_\R$ to pick an increasing modulus, $g$, for $f$. There is one subtlety here: we want $g$ to be increasing not only on the arithmetic degrees, but also on the Turing degrees. In other words, we want $g$ to be a function such that
\[
  g(x) \ge_T x \text{ and } y \ge_A g(x) \implies f(y) \ge_A x.
\]
This may look a little unintuitive: why not just require $g(x) \ge_A x$? The reason is that some of the lemmas we would like to invoke depend on finding a computable injective function on a pointed perfect tree and we cannot find a computable inverse for $g$ unless $g(x) \ge_T x$. It is possible to get around this difficulty by reproving our main lemmas with somewhat different hypotheses tailored to functions on the arithmetic degrees, but our approach of requiring $g(x)$ to compute $x$ seems easier and more general.

In any case, there is no problem with requiring $g(x)$ to compute $x$ because if $g$ is any modulus for $f$, we can always replace $g$ with $x \mapsto g(x)\oplus x$ to get a modulus which is increasing on the Turing degrees.

Now that we have $g$, we can proceed with the rest of the proof more or less unchanged. By Corollary~\ref{cor:framework_invertincreasing} we can find a pointed perfect tree, $T$, and a computable function $h$ defined on $[T]$ which is a right inverse for $g$ on $[T]$ (in this case, it does not matter whether $T$ is pointed in the sense of the Turing degrees or in the sense of the arithmetic degrees). For all $x \in [T]$, the definition of modulus implies that $f(x) \ge_A h(x)$ and Lemma~\ref{lemma:framework_continuousinverse} implies that $h(x) \oplus T \ge_T x$. Putting these together we have that $f(x)\oplus T \ge_A x$ on a cone of arithmetic degrees. Since $f$ is measure-preserving, it gets above $T$ on a cone and thus $f(x) \ge_A x$ on a cone, as desired.
\end{proof}

The only things about the degree structure that seem required to make this proof work are that it satisfies something like Martin's pointed perfect tree theorem\footnote{Note that this is not true of all degree structures. For example, Marks has observed that Martin's cone theorem fails for polynomial time Turing equivalence~\cite{marks2018universality}.} and its notion of reduction is reasonable enough to prove things like Corollary~\ref{cor:framework_invertincreasing} and Lemma~\ref{lemma:framework_continuousinverse}.

\subsubsection*{Order-preserving functions on other degree structures}

Unlike our results on measure-preserving functions, our results on order-preserving functions do not seem easy to generalize to other degree structures.

For order-preserving functions on the arithmetic degrees, for example, part 1 of Martin's Conjecture is simply false.

\begin{theorem}[Slaman and Steel]
\label{thm:gen_oparithmetic}
There is a function $f\colon \Cantor \to \Cantor$ which is order-preserving for arithmetic reducibility which is neither constant on a cone of arithmetic degrees nor above the identity on a cone of arithmetic degrees.
\end{theorem}

This theorem has not been published, but the construction is very similar to the counterexample to part 1 of Martin's Conjecture for the arithmetic degrees constructed in~\cite{marks2016martins}.

The theorem above also shows that the analog of Theorem~\ref{thm:op_mp} for the arithmetic degrees fails: not all functions which are order-preserving on the arithmetic degrees are measure-preserving on the arithmetic degrees. To see why, note that any counterexample to part 1 of Martin's Conjecture for order-preserving functions on the arithmetic degrees cannot be measure-preserving on the arithmetic degrees because, as observed above, part 1 of Martin's Conjecture does hold for such functions.

In contrast, we do not know whether part 1 of Martin's Conjecture for order-preserving functions holds for the hyperarithmetic degrees. On the one hand, we do not know how to extend the crucial Theorem~\ref{thm:op_mp} to the hyperarithmetic degrees. On the other hand, we also don't know of any counterexample. Resolving this seems like an interesting question.

\subsection{Non-invariant functions}
\label{sec:gen_non}
In the previous section, we saw that our proof of part 1 of Martin's Conjecture for measure-preserving functions is robust in the sense that it works in many different degree structures. In this section, we will see that it is robust in another sense: it works even for functions which are not Turing-invariant. As in the previous section, both the $\AD + \Unif_\R$ proof from section~\ref{sec:mp_proof} and the $\AD + \DC_\R$ proof from section~\ref{sec:mp_proof2} can easily be modified to work in this new setting, but we will just explain how to modify the $\AD + \Unif_\R$ proof since it's a bit simpler.

Before we actually give the proof, we will mention one reason why this sort of result is interesting. If we remove the requirement of Turing invariance from Martin's Conjecture then there are many counterexamples. In fact, practically every construction of classical computability theory gives rise to a counterexample. As a concrete example, consider the Friedberg jump inversion theorem. The proof of this theorem is quite constructive, and in fact produces a (non Turing-invariant) function $f\colon 2^\omega \to 2^\omega$ such that for each $x$, $f(x)' \equiv_T x\oplus 0'$. Thus $f$ is a regressive function which is neither constant on any cone nor ever above the identity.

An interesting feature of most of these classical constructions is that they produce reals which are, in some sense, generic. Sometimes this is even true in a precise technical sense. For example, if $f$ is the function produced by Friedberg jump inversion then $f(x)$ is always 1-generic relative to $x$.

It is reasonable to ask to what extent this is a necessary feature of such constructions. For example, is there a constructive proof of the jump inversion theorem that doesn't produce generic reals? Of course, this is a hard question to make precise, but we can at least note some ways that reals can fail to look generic. One feature of most types of generic reals is that they tend to avoid a cone---for example, if $g$ is a function such that $g(x)$ is always 1-generic relative to $x$ then $g$ avoids the cone above $0'$. And if $g(x)$ is instead always 1-random relative to $x$ then $g$ does not necessarily avoid the cone above $0'$ (there are 1-random reals in every degree above $0'$), but as soon as $x$ is above $0'$, it does. Thus any proof of jump inversion which produces a (not necessarily Turing-invariant) measure-preserving function would seem to be producing reals that are not generic.

The fact that part 1 of Martin's Conjecture for measure-preserving functions holds even without the restriction to Turing-invariant functions shows that there is no such measure-preserving jump inversion function. In other words, this feature of the proof of jump inversion seems to be inescapable.

Before describing how to modify the proof, we should clarify that when we say a non Turing-invariant function $f\colon \Cantor \to \Cantor$ is measure-preserving, we mean that for every $a \in \Cantor$, there is some $b$ such that
\[
  x \geq_T b \implies f(x) \geq_T a.
\]
In other words, exactly what we meant when we said a Turing-invariant function is measure-preserving.

\begin{theorem}[$\ZF + \AD + \Unif_\R$]
If $f \colon 2^\omega \to 2^\omega$ is measure-preserving (but not necessarily Turing-invariant) then $f(x) \ge_T x$ on a cone.
\end{theorem}

\begin{proof}
The proof is nearly identical to the proof for Turing-invariant functions. Let $g$ be an increasing modulus for $f$, and $h$ a computable function which inverts $g$ on a pointed perfect tree, $T$. The key point is that it follows from the definition of modulus that if $y$ is in $[T]$ and $x\ge_T y$ then $f(x) \ge_T h(y)$ and that this fact does not depend on $f$ being Turing-invariant.

To see why this is enough, suppose $x$ is large enough to be Turing equivalent to something in $T$ and also large enough such that $f(x) \ge_T T$. Let $y \in [T]$ be Turing equivalent to $x$. Then
\begin{align*}
  f(x) \ge_T h(y) \oplus T \ge_T y \equiv_T x.
\end{align*}
In other words, $f(x) \ge_T x$ for all large enough $x$.
\end{proof}

The fact that our proof of part 1 of Martin's Conjecture for measure-preserving functions still works for functions which are not Turing-invariant points to an interesting feature of the proof: it relies mostly on manipulating functions which are \emph{not} themselves Turing-invariant, even when the function $f$ is. The way that we chose the modulus, $g$, came with no guarantees that $g$ is Turing-invariant. What's more, it follows from Slaman and Steel's theorem on regressive functions that no inverse for $g$ can be Turing-invariant (otherwise it would yield a non-constant, regressive function on the Turing degrees).

\subsection{Ideal-valued functions}
\label{sec:gen_ideal}
We will now see that our result on measure-preserving functions is robust in yet another way. In particular, it holds for functions which take values in the set of Turing ideals.

\begin{definition}
A \term{Turing ideal} is a set of Turing degrees $\I$ such that
\begin{itemize}
\item \textbf{$\I$ is closed under Turing reducibility:} if $\degree{x} \in \I$ and $\degree{y} \le_T \degree{x}$ then $\degree{y} \in \I$.
\item \textbf{$\I$ is closed under finite joins:} if $\degree{x}_1,\ldots,\degree{x}_n \in \I$ then $\degree{x}_1\oplus \ldots\oplus \degree{x}_n \in \I$.
\end{itemize}
Also, $\I$ is \term{proper} if it is not equal to all of $\D_T$.
\end{definition}

\begin{notation}
Let $\Spec(\D_T)$ denote the set of Turing ideals.
\end{notation}

\begin{definition}
A function $\I \colon \Cantor \to \Spec(\D_T)$ is \term{measure-preserving} if for every $a \in \Cantor$, there is some $b \in \Cantor$ such that
\[
    x \geq_T b \implies a \in \I(x).
\]
\end{definition}

It is not hard to see that our proof of part 1 of Martin's Conjecture for measure-preserving functions can also be used to prove a similar result about ideal-valued functions.

\begin{theorem}[$\ZF + \AD + \Unif_\R$]
Suppose $\I\colon \Cantor \to \Spec(\D_T)$ is measure-preserving. Then for all $x$ on a cone, $x \in \I(x)$.
\end{theorem}

\begin{proof}
The proof is nearly identical to the proof for $\Cantor$-valued functions but requires a modification of the definition of ``modulus.'' In particular, call $f \colon \Cantor \to \Cantor$ a modulus for $\I$ if for all $a$,
\[
    x \geq_T f(a) \implies a \in \I(x).
\]
Call $f$ an increasing modulus if, in addition, for all $x$, $f(x) \geq_T x$. By $\Unif_\R$, $\I$ has an increasing modulus. By Corollary~\ref{cor:framework_invertincreasing}, we can invert $f$ on a pointed perfect tree to get a pointed perfect tree $T$ and a computable function $g$ defined on $[T]$ such that for all $x \in [T]$, $f(g(x)) = x$. Note that this implies that for all $x$, $g(x) \in \I(x)$.

Since $g$ is computable and injective on $[T]$, for all $x \in [T]$, $g(x)\oplus T \geq_T x$. If $x$ is of high enough Turing degree then $T \in \I(x)$ (since $\I$ is measure-preserving). Since Turing ideals are closed under Turing joins, this implies that for any such $x$,
\[
    g(x) \oplus T \in \I(x).
\]
Since Turing ideals are closed under Turing reducibility, this implies that $x \leq_T g(x)\oplus T$ is in $\I(x)$.
\end{proof}

In light of the above theorem, it seems natural to ask whether something similar holds for order-preserving ideal-valued functions. More precisely, say that a function $\I \colon \Cantor \to \Spec(\D_T)$ is \term{order-preserving} if for all $x, y \in \Cantor$,
\[
    x \leq_T y \implies \I(x) \subseteq \I(y).
\]
Suppose $\I \colon \Cantor \to \Spec(\D_T)$ is order-preserving. Must $\I$ either be constant on a cone or satisfy $x \in \I(x)$ on a cone? 

Unfortunately, this is false; a counterexample can be constructed from Slaman and Steel's counterexample to part 1 of Martin's Conjecture for order-preserving functions on the arithmetic degrees (see Theorem~\ref{thm:gen_oparithmetic} above). In particular, if $f\colon \Cantor \to \Cantor$ is a function which is order-preserving for the arithmetic degrees and which is neither constant on a cone nor above the identity on a cone (in the sense of arithmetic reducibility) then the ideal-valued function $\I$ defined by
\[
    \I(x) = \{\degreeof(y) \mid y \leq_A f(x)\}
\]
is a counterexample to the question above (where $\leq_A$ denotes arithmetic reducibility).

The existence of this counterexample provides an interesting contrast to a result by Slaman, who proved an analogue of part 2 of Martin's conjecture for Borel order-preserving functions $\I \colon \Cantor \to \Spec(\D_T)$ (see \cite{slaman2005aspects}, Theorem 3.1).
\subsection{$\ZFC$ counterexamples}
\label{sec:gen_zfc}
It is easy to construct counterexamples to Martin's Conjecture in $\ZFC$, even when the conjecture is restricted to some special class of functions, such as order-preserving functions. However, Slaman and Steel have observed in \cite{slaman1988definable} that if we alter the statement of Martin's Conjecture by replacing ``on a cone'' with ``cofinally'' then some special cases are provable in $\ZFC$. They give two examples of this phenomenon.
\begin{enumerate}
\item For any order-preserving function $f \colon \Cantor \to \Cantor$, if $f$ is above the identity on a cone then either $f(x) \equiv_T x$ on a cone or $f(x) \geq_T x'$ for cofinally many $x$.
\item For any regressive, measure-preserving function $f \colon \Cantor \to \Cantor$, $f(x) \equiv_T x$ for cofinally many $x$.
\end{enumerate}
In this section we will show that this does not happen for the main theorems of this paper: $\ZFC$ proves that there are counterexamples to part 1 of Martin's Conjecture for measure-preserving functions and order-preserving functions, even when we replace ``on a cone'' with ``cofinally'' in the conclusion of Martin's Conjecture.

\subsubsection*{Counterexample for measure-preserving functions}

We will first show that in $\ZFC$ we can construct a Turing-invariant function $f\colon \Cantor \to \Cantor$ which is measure-preserving and such that for all uncomputable $x$, $x \nleq_T f(x)$. The main idea of the construction is to write the Turing degrees as an increasing union of Turing ideals and define $f(x)$ to be a minimal upper bound for all the reals computable from $x$ which first show up in a strictly earlier ideal than $x$ itself. We will need one fact about Turing ideals.

\begin{lemma}
\label{lemma:gen_upperbound}
If $\I$ is a countable Turing ideal and $\degree{x}$ is an uncomputable Turing degree which is not contained in $\I$ then $\I$ has an upper bound which does not compute $\degree{x}$.
\end{lemma}

\begin{proof}
If $\I$ is empty then this is trivial. Otherwise, it follows from a theorem of Spector (see Exercise 6.5.12 of Soare's textbook~\cite{soare2016turing} for a proof) that $\I$ has an exact pair---i.e.\ Turing degrees $\degree{a}$ and $\degree{b}$ such that
\begin{itemize}
\item $\degree{a}$ and $\degree{b}$ are both upper bounds for $\I$
\item and if any Turing degree $\degree{y}$ is below both $\degree{a}$ and $\degree{b}$ then $\degree{y} \in \I$.
\end{itemize}
Since $\degree{x} \notin \I$, $\degree{x}$ cannot be computable from both $\degree{a}$ and $\degree{b}$ and thus at least one of them satisfies the desired conclusion.
\end{proof}

\begin{theorem}[$\ZFC$]
\label{thm:gen_mpzfc}
There is a Turing-invariant function $f \colon 2^\omega \to 2^\omega$ such that $f$ is measure-preserving and for all uncomputable $x$, $f(x)$ does not compute $x$.
\end{theorem}

\begin{proof}
Instead of defining a Turing-invariant function on the reals, we will define a function $F$ on the Turing degrees (which is equivalent under $\ZFC$).

Let $\seq{\I_\alpha}_\alpha$ be an increasing, well-ordered sequence of proper Turing ideals such that $\bigcup_\alpha \I_\alpha = \D_T$ and the sequence has no maximal element (such a sequence is easy to construct in $\ZFC$).

Given a Turing degree $\degree{x}$, let $\alpha_{\degree{x}}$ denote the least ordinal $\alpha$ such that $\degree{x} \in \I_\alpha$ and let $\I_{\degree{x}}$ denote
\[
  \I_{\degree{x}} = \{\degree{y} \in \D_T \mid \degree{y} \leq_T \degree{x} \text{ and for some $\beta < \alpha_{\degree{x}}$, $\degree{y} \in \I_\beta$}\}. 
\]
In other words, $\I_{\degree{x}}$ consists of those Turing degrees $\degree{y}$ which are computable from $\degree{x}$ and which show up in a strictly earlier ideal than $\degree{x}$. It is easy to check that $\I_{\degree{x}}$ is a countable (possibly empty) Turing ideal which does not contain $\degree{x}$.

If $\degree{x}$ is uncomputable then by Lemma~\ref{lemma:gen_upperbound}, $\I_{\degree{x}}$ has an upper bound which does not compute $\degree{x}$. Thus we can define $F$ as follows: if $\degree{x}$ is computable then set $F(\degree{x}) = \degree{0}$ and otherwise set $F(\degree{x})$ to be any upper bound for $\I_{\degree{x}}$ which does not compute $\degree{x}$.

By construction, $\degree{x} \nleq_T F(\degree{x})$ holds for all uncomputable $\degree{x}$. So we just need to check that $F$ is measure-preserving. To this end, fix a Turing degree $\degree{a}$ and we will show that $F$ is above $\degree{a}$ on a cone. Let $\alpha$ be the least ordinal such that $\degree{a} \in \I_\alpha$ and let $\degree{b}$ be some Turing degree such that $\degree{b} \notin \I_\alpha$ (which exists because $\I_\alpha$ is proper). Finally, let $\degree{c} = \degree{a}\oplus \degree{b}$.

We claim that $F$ is above $\degree{a}$ on the cone above $\degree{c}$. Let $\degree{x} \geq_T \degree{c}$. Since $\degree{x}$ computes $\degree{b}$, which is not in $\I_\alpha$ and since the sequence of ideals is increasing, $\alpha_{\degree{x}} > \alpha$. Since $\degree{x}$ also computes $\degree{a}$, $\degree{a} \in \I_{\degree{x}}$ and thus $F(\degree{x})$ computes $\degree{a}$.
\end{proof}

\subsubsection*{Counterexample for order-preserving functions}

We will now show that in $\ZFC$ we can construct a Turing-invariant function $f\colon \Cantor \to \Cantor$ which is order-preserving, not constant on any cofinal set and not above the identity on any cofinal set. Furthermore, the function we construct is not measure-preserving---in fact, the range is disjoint from a cone---and thus Theorem~\ref{thm:op_mp} also fails under $\ZFC$. The proof relies on a theorem of Sacks.

\begin{theorem}[$\ZFC$; Sacks~\cite{sacks1961suborderings}]
Every locally countable countable partial order of size $\omega_1$ can be embedded into the Turing degrees.
\end{theorem}

\begin{theorem}[$\ZFC$]
There is a Turing-invariant function $f \colon 2^\omega \to 2^\omega$ such that $f$ is order-preserving, not constant on any cofinal set and not above the identity on any cofinal set. Also, the range of $f$ is disjoint from a cone and hence $f$ is not measure-preserving.
\end{theorem}

\begin{proof}
We will once again just explain how to define a function on the Turing degrees which has the desired properties.

Let $(P, \le_P)$ be the partial order consisting of $\omega_1$, with its usual order, plus one point, $q$, which is incomparable to everything in $\omega_1$. Let $\pi$ be an embedding of $P$ into the Turing degrees.

For any $\degree{x}$, let $\alpha_{\degree{x}}$ be the least ordinal such that $\degree{x}$ does not compute $\pi(\alpha_{\degree{x}})$. Note that such an $\alpha_{\degree{x}}$ must exist since $x$ can only compute countably many degrees, but the range of $\pi$ is uncountable. Now define $F(\degree{x}) = \pi(\alpha_{\degree{x}})$.

\medskip\noindent\textbf{$F$ is order-preserving.} Let $\degree{x}$ and $\degree{y}$ be Turing degrees such that $\degree{x} \leq_T \degree{y}$. Note that for any $\alpha < \omega_1$, if $\degree{x}$ computes $\pi(\alpha)$ then so does $\degree{x}$. Hence $\alpha_{\degree{x}} \leq \alpha_{\degree{y}}$. Since $\pi$ is an embedding of partial orders, we have
\[
  F(\degree{x}) = \pi(\alpha_{\degree{x}}) \leq_T \pi(\alpha_{\degree{y}}) = F(\degree{y}).
\]

\medskip\noindent\textbf{$F$ is not constant on a cofinal set.} Suppose it is, with constant value $\degree{a}$. By definition of $F$, $\degree{a}$ must be equal to $\pi(\alpha)$ for some $\alpha$. But then for any $\degree{x}$ which computes $\degree{a}$, $\alpha_{\degree{x}}$ cannot be equal to $\alpha$ and thus since $\pi$ is an embedding, $F(\degree{x}) = \pi(\alpha_{\degree{x}}) \neq \pi(\alpha)$. Therefore on the cone above $\degree{a}$, $F(\degree{x}) \neq \degree{a}$, which contradicts the assumption that $F(\degree{x}) = \degree{a}$ on a cofinal set.

\medskip\noindent\textbf{The range of $F$ is disjoint from a cone.} In particular, the range of $F$ is disjoint from the cone above $\pi(q)$. This is because the range of $F$ is contained in the image of $\omega_1$ under $\pi$ and since $\pi$ is an embedding of partial orders, nothing in this image computes $\pi(q)$. Note that this also immediately implies that $F$ is not above the identity on any cofinal set.
\end{proof}

\section{Questions}

Throughout this paper we have mentioned some interesting questions raised by our work. For convenience we will now provide a list of these questions.

\begin{question}
Does part 1 of Martin's Conjecture hold for all order-preserving functions on the hyperarithmetic degrees?
\end{question}

\begin{question}
Is part 1 of Martin's Conjecture for measure-preserving functions provable in $\ZF$ when restricted to Borel functions?
\end{question}

\noindent A negative answer to questions~\ref{question:rk_below} and~\ref{question:rk_equiv} together would imply part 1 of Martin's Conjecture.

\begin{question}
\label{question:rk_below}
Is there any ultrafilter on the Turing degrees which is strictly below the Martin measure in the Rudin-Keisler order?
\end{question}

\begin{question}
\label{question:rk_equiv}
Is there any ultrafilter on the Turing degrees besides the Martin measure itself which is weakly Rudin-Keisler equivalent to the Martin measure?
\end{question}

\noindent Question~\ref{question:lebesgue_baire} is a special case of question~\ref{question:rk_below}.

\begin{question}
\label{question:lebesgue_baire}
Are the Lebesgue or Baire ultrafilters on the Turing degrees below the Martin measure in the Rudin-Keisler order?
\end{question}

\noindent Question~\ref{question:rk_above} has no direct bearing on Martin's Conjecture but seems intriguing.

\begin{question}
\label{question:rk_above}
Is there any ultrafilter on the Turing degrees which is strictly above Martin measure in the Rudin-Keisler order?
\end{question}

\bibliographystyle{plain}
\bibliography{main}

\begin{thebibliography}{10}

\bibitem{brattka2011weihrauch}
Vasco Brattka and Guido Gherardi.
\newblock Weihrauch degrees, omniscience principles and weak computability.
\newblock {\em The Journal of Symbolic Logic}, 76(1):143--176, 2011.

\bibitem{carroy2013quasi}
Rapha{\"e}l Carroy.
\newblock A quasi-order on continuous functions.
\newblock {\em The Journal of Symbolic Logic}, 78(2):633--648, 2013.

\bibitem{comfort1974theory}
W.~W. Comfort and S.~Negrepontis.
\newblock {\em The theory of ultrafilters}.
\newblock Die Grundlehren der mathematischen Wissenschaften, Band 211.
  Springer-Verlag, New York-Heidelberg, 1974.

\bibitem{dougherty2000how}
Randall Dougherty and Alexander~S. Kechris.
\newblock How many {T}uring degrees are there?
\newblock In {\em Computability theory and its applications ({B}oulder, {CO},
  1999)}, volume 257 of {\em Contemp. Math.}, pages 83--94. Amer. Math. Soc.,
  Providence, RI, 2000.

\bibitem{groszek1998basis}
Marcia~J. Groszek and Theodore~A. Slaman.
\newblock A basis theorem for perfect sets.
\newblock {\em The Bulletin of Symbolic Logic}, 4(2):204--209, 1998.

\bibitem{higuchi2023note}
Kojiro Higuchi and Patrick Lutz.
\newblock A note on a conjecture of sacks.
\newblock Preprint available at \url{http://arxiv.org/abs/2309.01876}, 2023.

\bibitem{jech2003set}
Thomas Jech.
\newblock {\em Set theory}.
\newblock Springer Monographs in Mathematics. Springer-Verlag, Berlin, 2003.
\newblock The third millennium edition, revised and expanded.

\bibitem{kechris1995classical}
Alexander~S. Kechris.
\newblock {\em Classical descriptive set theory}, volume 156 of {\em Graduate
  Texts in Mathematics}.
\newblock Springer-Verlag, New York, 1995.

\bibitem{kechris2021theory}
Alexander~S. Kechris.
\newblock The theory of countable borel equivalence relations.
\newblock Preprint available online at author's website, 2021.

\bibitem{kihara2021personal}
Takayuki Kihara.
\newblock Personal Communication, 2021.

\bibitem{koellner2010large}
Peter Koellner and W.~Hugh Woodin.
\newblock Large cardinals from determinacy.
\newblock In {\em Handbook of set theory. {V}ols. 1, 2, 3}, pages 1951--2119.
  Springer, Dordrecht, 2010.

\bibitem{lachlan1975uniform}
A.~H. Lachlan.
\newblock Uniform enumeration operations.
\newblock {\em J. Symbolic Logic}, 40(3):401--409, 1975.

\bibitem{lutz2023martins}
Patrick Lutz.
\newblock Martin's conjecture for regressive functions on the hyperarithmetic
  degrees.
\newblock Preprint available at \url{http://arxiv.org/abs/2306.05746}, 2023.

\bibitem{marks2018universality}
Andrew Marks.
\newblock The universality of polynomial time turing equivalence.
\newblock {\em Mathematical Structures in Computer Science}, 28(3):448–456,
  2018.

\bibitem{marks2020personal}
Andrew Marks.
\newblock Personal Communication, 2020.

\bibitem{marks2016martins}
Andrew Marks, Theodore~A. Slaman, and John~R. Steel.
\newblock Martin's conjecture, arithmetic equivalence, and countable {B}orel
  equivalence relations.
\newblock In Alexander~S. Kechris, Benedikt Löwe, and John~R. Steel, editors,
  {\em Ordinal definability and recursion theory: {T}he {C}abal {S}eminar.
  {V}ol. {III}}, volume~43 of {\em Lect. Notes Log.}, pages 493--519. Assoc.
  Symbol. Logic, Ithaca, NY, 2016.

\bibitem{martin1968axiom}
Donald~A. Martin.
\newblock The axiom of determinateness and reduction principles in the
  analytical hierarchy.
\newblock {\em Bull. Amer. Math. Soc.}, 74(4):687--689, 07 1968.

\bibitem{martin1969measurable}
Donald~A. Martin.
\newblock Measurable cardinals and analytic games.
\newblock {\em Fund. Math.}, 66:287--291, 1969/70.

\bibitem{martin1985purely}
Donald~A. Martin.
\newblock A purely inductive proof of {B}orel determinacy.
\newblock In {\em Recursion theory ({I}thaca, {N}.{Y}., 1982)}, volume~42 of
  {\em Proc. Sympos. Pure Math.}, pages 303--308. Amer. Math. Soc., Providence,
  RI, 1985.

\bibitem{montalban2019martins}
Antonio Montalb\'{a}n.
\newblock Martin's conjecture: a classification of the naturally occurring
  {T}uring degrees.
\newblock {\em Notices Amer. Math. Soc.}, 66(8):1209--1215, 2019.

\bibitem{pawlikowski2012decomposing}
Janusz Pawlikowski and Marcin Sabok.
\newblock Decomposing {B}orel functions and structure at finite levels of the
  {B}aire hierarchy.
\newblock {\em Ann. Pure Appl. Logic}, 163(12):1748--1764, 2012.

\bibitem{sacks1961suborderings}
Gerald~E. Sacks.
\newblock On suborderings of degrees of recursive unsolvability.
\newblock {\em Z. Math. Logik Grundlagen Math.}, 7:46--56, 1961.

\bibitem{slaman2005aspects}
Theodore~A. Slaman.
\newblock Aspects of the {Turing} jump.
\newblock In {\em Logic colloquium 2000. Proceedings of the annual European
  summer meeting of the Association for Symbolic Logic, Paris, France, July
  23--31, 2000}, pages 365--382. Wellesley, MA: A K Peters; Urbana, IL:
  Association for Symbolic Logic, 2005.

\bibitem{slaman1988definable}
Theodore~A. Slaman and John~R. Steel.
\newblock Definable functions on degrees.
\newblock In {\em Cabal Seminar 81--85}, pages 37--55, Berlin, Heidelberg,
  1988. Springer Berlin Heidelberg.

\bibitem{soare2016turing}
Robert~I. Soare.
\newblock {\em Turing computability}.
\newblock Theory and Applications of Computability. Springer-Verlag, Berlin,
  2016.

\bibitem{solecki1998decomposing}
S{\l}awomir Solecki.
\newblock Decomposing {B}orel sets and functions and the structure of {B}aire
  class {$1$} functions.
\newblock {\em J. Amer. Math. Soc.}, 11(3):521--550, 1998.

\bibitem{steel1982classification}
John~R Steel.
\newblock A classification of jump operators.
\newblock {\em The Journal of Symbolic Logic}, 47(2):347--358, 1982.

\bibitem{steel2009derived}
John~R. Steel.
\newblock The derived model theorem.
\newblock In {\em Logic Colloquium 2006}, Lecture Notes in Logic, page
  280–327. Cambridge University Press, 2009.

\bibitem{zapletal2004descriptive}
Jind\v{r}ich Zapletal.
\newblock Descriptive set theory and definable forcing.
\newblock {\em Mem. Amer. Math. Soc.}, 167(793):viii+141, 2004.

\end{thebibliography}

\end{document}